\documentclass[11pt]{article}
\usepackage[T1]{fontenc}
\usepackage[utf8]{inputenc}
\usepackage{authblk}
\usepackage{graphicx, bbding}
\usepackage{ mathrsfs }
\usepackage{psfrag}
\usepackage{subfigure}
\usepackage{amsmath}
\usepackage{marginnote}
\usepackage[left = 1.0in, right = 1.0in, top = 1.0in, bottom = 1.0in]{geometry}
\RequirePackage{amsthm}
\RequirePackage[colorlinks]{hyperref}
\RequirePackage{ txfonts}
\def\R{\mathbb{R}}
\def\tr{\mathrm{tr}}
\def\D{\mathcal{D}}

\def\ind{\perp\!\!\perp}
\def\G{\mathcal{G}}
\def\G{\mathcal{G}}

\def\P{\mathrm{P}}
\def\Q{\mathrm{Q}}
\def\W{\mathcal{W}}

\theoremstyle{plain}
\newtheorem{Theorem}{Theorem}[section]

\newtheorem{Notation}{Notation}[section]

\newtheorem{lemma}{Lemma}[section]
\newtheorem{proposition}{Proposition}[section]
\theoremstyle{definition}
\newtheorem{definition}{Definition}[section]
\theoremstyle{remark}
\newtheorem*{notation}{Notation}
\numberwithin{equation}{section}
\newtheorem{Ex}{Example}[section]
\newtheorem{Rem}{Remark}[section]
\title{The Letac-Massam conjecture and existence of high dimensional Bayes estimators for Graphical Models}
\author[1]{Emanuel Ben-David\thanks{ehb2126@columbia.edu}}
\author[2]{Bala Rajaratnam \thanks{brajarat@stanford.edu}}
\affil[1]{Department of Statistics, Columbia University}
\affil[2]{Department of Statistics, Stanford University}
\date{}

\begin{document}
\maketitle

\begin{abstract}

The Wishart distribution, defined on the open convex positive definite cone, plays a central role in multivariate analysis and multivariate distribution theory. Its domain of integrability is often referred to as the Gindikin set. In recent years, a variety of useful extensions of the Wishart have been proposed in the literature for the purposes of studying Markov random fields / graphical models. In particular, generalizations of the Wishart, referred to as Type I and Type II Wishart distributions, have been introduced by Letac and Massam (\emph{Annals of Statistics} \cite{L07}) and play important roles in both frequentist and Bayesian inference for Gaussian graphical models. These distributions have been especially useful in high-dimensional settings due to the flexibility offered by their multiple shape parameters. The domain of integrability of these graphical Wisharts are however not fully identified, despite its critical role in determining existence of high dimensional Bayes estimators and specifying diffuse proper priors for Gaussian graphical models. Moreover, these graphical Wisharts also serve as statistical models in their own right for matrix-variate distributions defined on sparse convex subsets of the cone. Besides its statistical motivation, understanding the domain of integrability is also of independent mathematical interest as these graphical Wisharts are extensions of the Gamma function on sparse manifolds. In this paper we resolve a long-standing conjecture of Letac and Massam (LM) concerning the domains of the multi-parameters of graphical Wishart type distributions. This conjecture, posed in \emph{Annals of Statistics}, also relates fundamentally to the existence of Bayes estimators corresponding to these high dimensional priors. To achieve our goal, we first develop novel theory in the context of probabilistic analysis of graphical models. Using these tools, and a recently introduced class of Wishart distributions for directed acyclic graph (DAG) models, we proceed to give counterexamples to the LM conjecture, thus completely resolving the problem. Our analysis also proceeds to give useful insights on graphical Wishart distributions with implications for Bayesian inference for such models.

\end{abstract}

\section{Introduction}\label{intro}
\noindent
Inference for graphical models is a topic of contemporary interest, and in this regard, various tools for inference have been proposed in the statistics literature, including establishing sufficient and/or necessary conditions for existence of high dimensional estimators. One important contribution in the area are the families of Type I and Type II graphical Wishart distributions introduced by Letac and Massam (LM) \cite{L07}. The families of graphical Wishart type distributions of Letac-Massam have the distinct advantage of being standard conjugate for Gaussian graphical models, have attractive hyper Markov properties, and have multiple shape parameters. This is in contrast with the classical Wishart distribution which has just one shape parameter that is restricted to the one dimensional Gindikin set - see \eqref{eq:gindikin_set}. These multi-parameter graphical Wishart distributions are therefore useful for flexible high dimensional inference \cite{R08}, and have also been used as flat conjugate priors for objective Bayesian inference. Since the domain of integrability of these high dimensional priors are not fully identified, it is not clear when these distributions yield proper priors. The LM conjecture aims to address this question formally. The LM conjecture is also critical for understanding when these priors lead to well-defined Bayes estimators, since this is not always guaranteed in high dimensional, sample starved settings. In this sense resolving the LM conjecture can be viewed as a Bayesian analogue of the frequentist problem of identifying sufficient and necessary conditions for the existence of the maximum likelihood estimator for Gaussian graphical models. 

The primary goal of this paper is therefore to resolve a conjecture of Letac and Massam (henceforth the LM conjecture) which concerns identifying the parameter sets for the families of the so-called Type I and Type II Wishart distributions. A definitive solution to the LM conjecture has remained elusive to the graphical models community ever since it was formally posed by Letac and Massam about ten years ago. The conjecture also has deep and profound connections to Gindikin's result \cite{F94, G75} on the region of integrability of the p-variate Gamma function. This domain is referred to as the Gindikin set and is given as follows:
\begin{equation}\label{eq:gindikin_set}
\Delta:=\left\{1, \frac{1}{2},\frac{3}{2},\ldots, \frac{p-1}{2} \right\}\cup\left(\frac{p-1}{2}, +\infty\right).
\end{equation}

Though the main goal of this paper is to resolve the LM conjecture, we note that understanding the domain of integrability of these graphical Wishart distributions is important for two other reasons beyond Bayesian inference and model selection: 1) These two classes of distributions also serve as statistical models in their own right for matrix-variate distributions defined on sparse subsets of the cone, and 2) The integrals of these graphical Wishart densities are extensions of the gamma and multivariate gamma functions on sparse manifolds. Thus understanding the domains of integrability of these graphical Wishart distributions is of independent mathematical interest that is closely linked to generalizations of the Gindikin set. 

In what follows we shall employ the notation introduced in the work of Letac and Massam \cite{L07}. The Type I Wishart and Type II Wishart are defined,  respectively, on the cones   $\Q_{\G}$ and  $\P_{\G}$  associated with a decomposable graph $\G$, i.e., an undirected graph that that has no induced cycle of length greater than or equal to four. These cones naturally arise as the set of covariance and inverse-covariance parameters for a Gaussian undirected graph model over $\G$. i.e., the family of multivariate Gaussian distributions that obey the pairwise or global Markov property with respect to $\G$  \cite{L96}.  It is well known that if the vertices of $\G$  are labeled $1,2,\ldots, p$, then a p-variate Gaussian distribution $\mathcal{N}_{p}\left(0, \Sigma\right)$  obeys the global Markov property with respect to $\G$  if $\Sigma_{ij}^{-1}=0$  whenever there exists no edge between $i$ and $j$. This property gives a simple characterization of the associated inverse-covariance matrices, i.e., the elements of the cone $\P_{\G}$. The cone $\Q_{\G}$ is the dual cone of $\P_{\G}$ and its elements are incomplete covariance matrices where only the entries along the edges of $\G$ are specified, and the rest of the entries are unspecified. However, the specified entries are also the only functionally independent entries of the covariance matrix parameter, and uniquely determine the rest of the entries (the unspecified entries that is). In particular, the space of covariance matrices for the Gaussian inverse-covariance graph model over $\G$  can be identified with the cone $\Q_{G}$.  When $\G$ is complete, i.e., in the full model, Type I and Type II Wishart distributions are identical to the classical Wishart distribution. Moreover by restricting the multi-parameters to a specific one dimensional space, these distributions reduce to the hyper Wishart distribution introduced by Dawid and Lauritzen \cite{D93} and the G-Wishart defined by Roverato in  \cite{RO00} respectively (see \cite{L07} for more details). Although having multiple shape parameters allow the Type I Wishart and Type II Wishart distributions to be more flexible as prior distributions, there is a trade-off: the sets of multi-parameters are not completely identified.

In an attempt to identify the set of multi-parameterss of the Type I Wisharts, denoted by $\mathcal{A}$, and that of Type II Wisharts, denoted by $\mathcal{B}$, in \cite[ Section 3.3]{L07} Letac and Massam first consider the case when $\G$ is homogeneous, i.e., $\G$ is decomposable and has no induced paths of length greater than or equal to 4. When $\G$ is homogeneous Letac and Massam are able to completely identify $\mathcal{A}$  and $\mathcal{B}$ and, furthermore, give algebraic expressions for the elements of both sets. If $\G$ is non-homogeneous, however, in \cite[ Section 3.4]{L07} the authors are able to only partially identify the sets $\mathcal{A}$ and $\mathcal{B}$. More specifically, for each perfect order  $\mathcal {P}$ of $\G$, they identify a subset $\mathcal{A}_{\mathcal{P}}$ of $\mathcal{A}$  and a subset $\mathcal{B}_{\mathcal{P}}$ of $\mathcal{B}$. The authors then proceed to conjecture that  $\mathcal{A}$  and  $\mathcal{B}$  are indeed the union of  $\mathcal{A}_{\mathcal{P}}$ and $\mathcal{B}_{\mathcal{P}}$  over all perfect orders of the cliques of $\mathcal{G}$, respectively. They demonstrate that the conjecture holds when $\G$ is the 4-path, the simplest non-homogeneous decomposable graph,  $\stackrel{\tiny{1}}{\bullet} - \stackrel{\tiny{2}}{\bullet} - \stackrel{\tiny{3}}{\bullet} - \stackrel{\tiny{4}}{\bullet}$. They note that a similar calculation for the 5-path appears insurmountable.

On a different route, but motivated by the recent work of Letac and  Massam \cite{L07} and Rajaratnam et al. \cite{R08} for concentration graph models, and Khare and Rajaratnam \cite{KR11} for covariance graph models, the authors of this paper undertook a parallel analysis in \cite{BR11} for directed acyclic graph models, abbreviated DAG models, or Bayesian networks. In \cite{BR11}, we introduce a new class of multi-parameter Wishart type distributions, useful for Bayesian inference for Gaussian DAG models. One of its advantages is that the framework in \cite{BR11} applies to all directed acyclic graph models and not just the narrower class of perfect DAGs. Furthermore the normalizing constant for these DAG Wisharts is available in closed form for all DAGs. It is also well-known fact that the family of inverse-covariance graph model over a decomposable graph $\G$ is Markov equivalent to the family of DAG models over a perfect DAG version of $\G$. As we shall demonstrate later, this, in particular, implies that both the Type II Wisharts of Letac-Massam in \cite{L07} and the DAG Wisharts in \cite{BR11} are indeed defined on the same cone $\P_{\G}$. Therefore, a relevant question is how the functional form and the multi-parameter set of the Type II Wishart density compare with those of the DAG Wishart. A similar comparison arises between the Type I Wisharts of Letac-Massam and the Riez distributions for decomposable graphs introduced by Andersson and Klein in \cite{A10}. Such comparisons shed light on the LM conjecture since the domains of integrability of the DAG Wisharts are fully specified in \cite{BR11}. In this paper we develop tools which allows a careful comparison of these two types of Wisharts on $\P_{\G}$ and $\Q_{\G}$, leading to counterexamples, which in turn can then be used to conclude that the LM conjecture does not hold in general.
  
  The primary key to resolving the part of the LM conjecture that concerns the Type II  Wishart is Theorem \ref{thm:comparison} of this paper. In this theorem we show that for any non-homogeneous decomposable graph $\mathcal{G}$ there exists a perfect order  $\mathcal{P}$ and  a perfect DAG version of $\mathcal{G}$, associated with this order,  such that the Type II Wishart distribution on  $\mathcal{B}_{\mathcal{P}}$ is a special case of the DAG Wishart distribution. Using this observation, and depending on the perfect DAG version of the underlying graph $\G$, we derive a condition in Proposition \ref{prop:B1}, which when satisfied, can lead to counterexamples to the LM conjecture. We then proceed to present two graphs (with their respective perfect DAG versions) where the stated condition in Proposition \ref{prop:B1} is satisfied and lead to counterexamples to the LM conjecture. The counterexample to the other part of the LM conjecture concerning the Type I Wishart distribution is given after Proposition \ref{prop:A1} where  we prove that the same condition as  that in Proposition \ref{prop:B1} can lead  to resolving the  LM conjecture. In addition to disproving the LM conjecture, we also prove that not only for non-homogeneous decomposable graphs, but also for homogeneous graphs, the family of Type II Wisharts are a subclass of the family of DAG Wisharts.

The organization of the paper is as follows. In \S \ref{sec:pre} we recall some fundamental notation and concepts in graphical models and, in particular, for Gaussian undirected graphical models. In \S \ref{sec:wishart_for_decomposable graph} we provide the reader with definition of Type I  and Type II Wishart distributions and formally state the Letac-Massam conjecture. In \S \ref{subsec:gaussian_over_D} and \S \ref{sub:dag_wishart} we give a short introduction to Guassian DAG models and the families of DAG Wisharts. The main results of the paper are presented in the ensuing four sections. In  \S \ref{sub:comparison_non_hom} and  \S \ref{sub:typeII_dag_wish_hom}, we develop tools which enable a detailed comparison between Type II Wisharts on one hand, and on the other hand, DAG Wisharts for the corresponding DAG versions of the associated undirected graphs. Moreover, tools are developed for comparisons of both decomposable and homogeneous Type II Wisharts to their DAG Wishart counterparts. Using the tools developed in \S \ref{sub:comparison_non_hom} and \S  \ref{sub:typeII_dag_wish_hom}, we formally resolve the Letac-Massam conjecture in  \S \ref{sub:LM_partb} and  \S \ref{sub:LM_parta} by providing counterexamples.

\section{Preliminaries} \label{sec:pre}

We now introduce some preliminaries on graph theory and graphical models. This section closely follows the notation and exposition given in \cite{BR11} and \cite{BR12}.

\subsection{Graph theoretic notation and terminology}\label{sub:graph}
A graph $\G$ is a pair of objects $\left(V, E\right)$, where $V$ and $E$ are two disjoint finite sets representing, respectively, the vertices and the edges of $\G$. An edge $e\in E$ is said to be undirected if $e$ is an unordered pair $\{v,v'\}$, or directed if $e$ is an ordered pair $(v,v')$ for some $v,v'\in V$. Now a graph is said to be undirected if all its edges are undirected, and directed if all its edges are directed. A directed edge $\left(v,v'\right)\in E$ is denoted by $v\rightarrow v'$. When $v\rightarrow v'$ and $v\neq v'$ we say that $v $ is a parent of $v'$, and $v'$ is a child of $v$. The set of parents of $v$ is denoted by $\mathrm{pa}\left(v\right)$, and the set of children of $v$ is denoted by $\mathrm{ch}\left(v\right)$. The family of $v$ is $\mathrm{fa}\left(v\right):=\mathrm{pa}\left(v\right)\cup \left\{v\right\}$. For an undirected edge $\left\{v,v'\right\}\in E$ the vertex $v$ is said to be a neighbor of $v'$, or $v'$ a neighbor of $v$, if $v\neq v'$. The set of all neighbors of $v$ is denoted by $\mathrm{ne}\left(v\right)$. In general two distinct vertices are said to be adjacent, denoted by $v\sim v'$, if there exists  either a directed or an undirected edge between them. A loop in $\G$ is an ordered pair $\left(v,v\right)$, or an unordered pair $\left\{v,v\right\}$ in $E$. For ease of notation, in this paper we shall always assume that the edge set of each graph contains all the loops, however, we draw the graph without the loops. \\
 A graph $\G'=\left(V', E'\right)$ is a subgraph of  $\G=\left(V, E\right)$, denoted by  $\G'\subseteq \G$, if  $V'\subseteq V$  and  $ E'\subseteq  E$. In addition, if  $\G'\subseteq \G$  and  $ E'=V'\times V'\cap  E$, we say that $\G'$ is an {induced}  subgraph of  $\G$. We shall consider only induced subgraphs in what follows. For a set $A\subseteq V$, the subgraph  $\G_A=\left(A, A\times A\cap  E\right)$ is said to be the graph induced by $A$. A graph  $\G$  is called complete if every pair of vertices are adjacent. A subset $A\subseteq V$  is said to be a clique, if the induced subgraph  $\G_A$ is a complete subgraph of $\G$ that is not contained in any other complete subgraphs of $\G$. A path in $\G$ of length $n\geq 1$ from a vertex $v$ to a vertex $v'$ is a finite sequence of distinct vertices $v_{0}=v,\ldots, v_{n}=v'$ in $V$ such that $\left(v_{k-1}, v_k\right)$  or $\left\{ v_{k-1}, v_k\right\}$ are in $E$ for each $k=1,\ldots, n$. A path is said to be directed if at least one of the edges is directed. We say  $v$  {leads} to $v'$, denoted by  $v\mapsto v'$, if there is a directed path from  $v$  to  $v'$. A graph $\G=\left(V, E\right)$ is said to be connected if for any pair of distinct vertices $v, v'\in V$ there exists a path between them. An  $n$-{cycle} in  $\G$  is a path of length  $n$  with the additional requirement that the end points are identical. A {directed}  $n$-cycle is defined accordingly. A graph is {acyclic} if it does not have any cycles. An acyclic directed  graph, denoted by  DAG (or ADG), is a directed graph with no cycles of length greater than 1.
\begin{notation}
Henceforth in this paper, we denote an undirected graph by $\G=\left(V, E\right)$ and a DAG by $\D=\left(V, F\right)$. Also, otherwise stated, we always assume that the vertex set $V=\{1,2,\ldots,p\}$.
\end{notation}
\indent The undirected version of a DAG $\D=\left(V, F\right)$, denoted by $\D^{\mathrm{u}}=\left(V,  F^{\mathrm{u}}\right)$,  is the undirected graph obtained by replacing all the directed edges of  $\D$  by undirected ones. An immorality in $\D$  is  an induced subgraph of the from  $v\rightarrow v' \leftarrow v''$.  Moralizing an immorality entails adding an undirected edge between the pair of parents that have the same children. Then the moral graph of  $\D$, denoted by  $\D^{\mathrm{m}}=\left(V, F^{\mathrm{m}}\right)$, is the undirected graph obtained by first  moralizing each  immorality of  $\D$  and then making the undirected version of the resulting graph. Naturally there are DAGs which have no immoralities and this leads to the following definition.
\begin{definition}
A DAG   $\D$ is said to be perfect if it has no immoralities; i.e., the parents of all vertices are adjacent, or equivalently if the set of parents of each vertex induces a complete subgraph of  $\D$ .
\end{definition}
\indent Given a directed acyclic graph (DAG), the set of {ancestors} of a vertex $v$, denoted by $\mathrm{an}\left(v\right)$, is the set of those vertices $v''$  such that $v''\mapsto v$. Similarly, the set of descendants of  a vertex $v$, denoted by  $\mathrm{de}\left(v\right)$, is the set of  those  vertices  $v'$  such that  $v\mapsto v'$.  The   set of {non-descendants } of  $v$  is  $\mathrm{nd}\left(v\right)=V\setminus\left(\mathrm{de}\left(v\right)\cup \left\{v\right\}\right)$.  A set  $A\subseteq V$  is said to be ancestral when $A$  contains  the parents of its members. The smallest ancestral set containing a set $B\subseteq V$ is denoted by $\mathrm{An}\left(B\right)$. 
\subsection{Decomposable and homogeneous graphs}\label{sub:decomp_ug}
Let $\G$  be a decomposable graph. The reader is referred to Lauritzen \cite{L96}  for all the common notions of decomposable graphs that we will use here. One such important notion is that of a perfect order of the cliques. Every decomposable graph admits a  perfect order of its cliques. Let
$\left(C_1,\cdots, C_{r}\right)$ be one such perfect order of the cliques of the graph $\G$. The history
for the graph is given by $H_1 = C_1$ and
\[
H_{j} = C_{1} \cup C_2 \cup \cdots \cup C_{j}, \; \; j = 2, 3, \cdots, r.
\]      

\noindent
The separators of the graph are given by
\[
S_{j} = H_{j-1} \cap C_j, \; \; j = 2, 3, \cdots, r.
\]
\noindent
The residuals are defined as follows:
\[
R_j = C_j \setminus H_{j-1} \mbox{ for } j = 2, 3, \cdots, r.
\]
\noindent
Generally, we will denote by $\mathscr{C}_{\G}$ the set of cliques of a graph and by $\mathscr{S}_{\G}$ its set of separators. Let $r' \leq r - 1$ denote the number of distinct separators and $\nu \left(S\right)$ denote the multiplicity of $S$, i.e., the number of $j$ such that $S_j = S$. 

Decomposable (undirected) graphs and (directed) perfect graphs have a deep connection. If $\G$ is decomposable, then there exists a perfect DAG version of $\G$, i.e., a perfect DAG $\D$ such that $\D^{\mathrm{u}}=\G$. On the other hand, the undirected version of a perfect DAG is necessarily decomposable \cite{GO04, L96}.

A decomposable graph $\G$ is said to to be homogeneous if for any two adjacent vertices $i, j$ we have
\begin{equation}\label{eq:homogeneous}
\mathrm{ne}\left(j\right)\cup\left\{ j\right\}\subseteq \mathrm{ne}\left(i\right)\cup\left\{ i\right\}\:\;\text{or}\;\: \mathrm{ne}\left(i\right)\cup\left\{ i\right\}\subseteq \mathrm{ne}\left(j\right)\cup\left\{ j\right\}.
\end{equation}
The reader is referred to Letac and Massam \cite{L07} for all the common notions of homogeneous graphs.

\subsection{Undirected Gaussian Graphical Models}\label{sub:gaussian_over_G}
 Let $\G=\left(V,E \right)$ be a undirected graph with $V=\left\{1,\ldots, p\right\}$ and let $\mathbf{X}=\left(X_{1},\ldots, X_{p}\right)^{\top}$ be a random vector in $\R^{p}$ such that $\mathbf{X}\sim\mathcal{N}_{p}\left(0, \Sigma\right)$, i.e.,  $\mathbf{X}$ has a $p$-variate Gaussian distribution with mean zero and covariance $\Sigma$. The covariance matrix $\Sigma$ is assumed to be positive definite (written as $\Sigma \succ 0$) with inverse-covariance matrix (also said to be the precision or  concentration matrix). $\Omega:=\Sigma^{-1}$. Now for any two vertices $i,j\in V$ 
\begin{equation*}
X_{i}\ind X_{j}|X_{V\setminus \left\{i,j\right\}}\Longrightarrow  \Omega_{ij}=0.
\end{equation*}
A simple proof of this well-known fact can be found in \cite[section 5.1]{L96}. In particular, the distribution $\mathcal{N}_{p}\left(0, \Sigma\right)$ is said to be a  Markov random field over $\G$ if
\begin{equation}\label{eq:concentration_graph}
\left\{i,j\right\}\notin E\implies \Omega_{ij}=0.
\end{equation}
Now let $\mathscr{N}\left(\G\right)$ denote the family of all $p$-variate Gaussian distributions $\mathcal{N}_{p}\left(0, \Sigma\right)$ that are Markov random fields over $\G$. Note that Equation \eqref{eq:concentration_graph} provides an easy description of the elements of $\mathscr{N}\left(\G\right)$ in terms of the pattern of zeros in the associated inverse-covariance matrices. Subsequently, $\mathscr{N}\left(\G\right)$ is said to be the Gaussian concentration graph model over $\G$. The set of covariance matrices
\[
\mathrm{PD}_{\G}:=\left\{\Sigma\succ 0 : \mathcal{N}_{p}\left(0,\Sigma\right)\in \mathscr{N}\left(\G\right)\right\}
\]
 is the standard parameter set for $\mathscr{N}\left(\G\right)$ . In light of Equation \eqref{eq:concentration_graph} the distributions in the exponential family $\mathscr{N}\left(\G\right)$ can be parametrized by the canonical parameter $\Omega=\Sigma^{-1}$ which lives in the space of inverse-covariance matrices defined as follows:
 \[
 \P_{\G}:=\left\{
\Omega\succ 0 : \Omega_{ij}=0\: \text{whenever}\: \left\{i,j\right\}\notin E\right\}.
\]
Let $\mathrm{S}_p$ denote the set of $p\times p$ symmetric matrices. Then 
 \begin{align*}
 \mathrm{Z}_{\G}&:=\left\{ A\in\mathrm{S}_p :\:  A_{ij}=0 \:\text{whenever  $\left\{i,j\right\}\notin E$}\right\};\\
 \mathrm{I}_{\G}&:=\left\{\Gamma=(\Gamma_{ij})\in \R^{E}: \: \Gamma_{ij}=\Gamma_{ij} \: \text{for every $\{i,j\}\in E$} \right\}.
 \end{align*}
We call each element in $\mathrm{I}_{\G}$ a $\G$-incomplete matrix. One can easily check that $\mathrm{I}_{\G}$ is a real linear space, isomorphic to $\mathrm{Z}_{\G}$, of dimension $|E|$. A $\G$-incomplete matrix $\Gamma$ is said to be partial positive definite over  $\G$ if for each clique $C\in\mathscr{C}_{\G}$ the $|C|\times |C|$ matrix $\Gamma_{C}=\left(\Gamma_{ij}\right)_{i,j\in C}$ is positive definite. Note that for any decomposable graph $\G$, the set of partial positive definite matrices over $\G$, denoted by $\Q_{\G}$, is the dual cone of the (open convex) cone $\P_{\G}$ \cite{L07}. When $\G$ is decomposable Grone et
al. in \cite{G84} prove that each $\Gamma\in \Q_{\G}$ can be completed to a unique positive definite matrix $\Sigma=\Sigma\left(\Gamma\right)\in \mathrm{PD}_{\G}$. This means that $\Sigma$ is the only element in $\mathrm{PD}_{\G}$ with the property that $\Sigma_{ij}=\Gamma_{ij}$, for each $\left\{i,j\right\}\in E$. If $\Sigma^{E}$ denotes an element of $\Q_{\G}$
with the unique positive definite completion $\Sigma$ in $\mathrm{PD}_{\G}$, then Grone et al. \cite{G84} explicitly
provide a bijective mapping  $\Sigma^{E}\mapsto \Sigma: \Q_{\G}\rightarrow \mathrm{PD}_{\G}$. If we compose this mapping with the inverse mapping $\Sigma\mapsto \Sigma^{-1}: \mathrm{PD}_{\G}\rightarrow \P_{G}$, then we obtain the bijective mapping $\Sigma^{E}\mapsto \Sigma^{-1}: \Q_{\G}\rightarrow \P_{\G}$. The corresponding inverse mapping is given as $\Omega\mapsto \Omega^{-E}: \P_{\G}\rightarrow\Q_{\G}$  where $\Omega^{-E}:=\left(\Omega^{-1}\right)^{E}$. We shall frequently invoke these mappings in subsequent sections.
 
 
 \section{The Letac-Massam Wishart type distributions for decomposable graphs}\label{sec:wishart_for_decomposable graph}
Henceforth in this paper, we assume that $\G=\left(V,E\right)$ is a decomposable graph and the vertices are labeled $1,2,\ldots, p$. The primary goal of this section is to provide the reader with an overview of the families of Wishart-Type I and Wishart-Type II distributions introduced in \cite{L07}. At the end of this section we shall formally state the LM conjecture concerning the domains of the multi-parameters for these distributions.
 
\subsection{Markov ratios and corresponding measures on $\Q_{\G}$ and $\P_{\G}$}
 Let $C_{1},\ldots, C_{r}$ be a perfect order of the cliques of $\G$ and let $\left(S_{2}, \ldots, S_{r}\right)$ be the corresponding sequence of  separators, with possible repetitions. For each $\alpha\in \R^{r}, \; \beta \in \R^{r-1}$ and $\Sigma^{E}\in \Q_{\G}$, the Markov ratio
$H_{\G}\left(\alpha,\beta, \Sigma^{E}\right)$ is defined as follows:
 \[
H_{\G}\left(\alpha,\beta,
\Sigma^{E}\right):=\dfrac{\prod_{i=1}^{r}\det\left(\Sigma_{C_{i}}\right)^{\alpha_{i}}}{\prod_{i=2}^{r}\det\left(\Sigma_{S_{i}}\right)^{\beta_{i}}}.
 \] 
Let $c:=\left(c_{1},\ldots,c_{r}\right)$ and $s:=\left(s_{2},\ldots,s_{r}\right)$ where $c_{i}:=|C_{i}|$ and $s_{i}:=|S_{i}|$,
respectively. Moreover, let $d\Sigma^{E}$ denote Lebesgue measure on $\Q_{\G}$ \footnote{More precisely, $d\Sigma^{E}$ is the standard Lebesgue measure on $\mathrm{I}_{\G}$ restricted to the open set $\Q_{\G}$.}. Then 
 \begin{equation}\label{eq:mu}
\mu_{\G}\left(d\Sigma^{E}\right):= H_{\G}\left(-\left(c+1\right)/2,-\left(s+1\right)/2,\Sigma^{E}\right)d\Sigma^{E}
 \end{equation}
is a measure on $\Q_{\G}$. The image of $\mu_{\G}$ under the mapping $\Sigma^{E}\mapsto \Sigma^{-1}: \Q_{\G}\rightarrow \P_{\G}$ is a measure on $\P_{\G}$ given by
  \begin{equation}\label{eq:nu}
\nu_{\G}\left(d\Omega\right):=  H_{\G}\left(\left(c+1\right)/2, \left(s+1\right)/2, \Omega^{-E}\right)d\Omega,  
\end{equation}
where $d\Omega$ is Lebesgue measure on $\P_{\G}$ \cite{L07}.
  
\subsection{Type I \& II Wishart distributions} 
We now introduce the Type I and Type II Wishart distributions from \cite{L07}. The Type I Wishart is a distribution defined on the cone $\Q_{\G}$.
The non-normalized density of this distribution is given by
  \begin{equation*}\label{eq:unnormalized_wishart}
\omega_{\Q_{\G}}\left(\alpha,\beta,U^{E},d\Sigma^{E}\right):=\exp\left\{-\tr\left(\Sigma U^{-1}\right)\right\}H_{\G}\left(\alpha,
\beta, \Sigma^{E}\right)\mu_{\G}\left(d\Sigma^{E}\right),
  \end{equation*}
where $\left(\alpha,\beta\right)\in \R^{r}\times \R^{r-1}$ denotes the multi-shape parameter and  $U^{E}\in \Q_{\G}$ is the scale parameter. The normalized version of $\omega_{\Q_{\G}}$, denoted by $\W_{\Q_{\mathcal{G}}}$,  is defined for pairs of $\left(\alpha, \beta\right)$ such that  for every $U^{E}\in \Q_{\G}$
  \begin{equation}\label{eq:A1}
\tag{A1}\int_{\Q_{\G}}\omega_{\Q_{\G}}\left(\alpha,\beta,U^{E}, d\Sigma^{E}\right)< \infty \:\;\text{and} 
\end{equation}
\vspace{-.4cm}
\begin{equation}\label{eq:A2}
\tag{A2}\int_{\Q_{\G}}\omega_{\Q_{\G}}\left(\alpha,\beta,U^{E}, d\Sigma^{E}\right)/H_{\G}\left(\alpha,\beta, U^{E}\right)\;\;\text{ is functionally independent of $U^{E}$}.
  \end{equation}
The Type II  Wishart is a distribution on the cone $\P_{\G}$ with the non-normalized density 
  \[
  \omega_{\P_{\G}}\left(\alpha,\beta,U^{E}, d\Omega\right):=\exp\left\{-\tr\left(\Omega U\right)\right\}H_{\G}\left(\alpha, \beta,
\Omega^{-E}\right)\nu_{\G}\left(d\Omega\right).
\]
Similarly, the normalized version of $ \omega_{\P_{\G}}$, denoted by $\W_{\P_{\mathcal{G}}}$,  is defined for pairs of $\left(\alpha, \beta\right)$ such that for every $U^{E}\in \Q_{\G}$
\begin{equation}\label{eq:B1}
\tag{B1}\int_{\P_{\G}}\omega_{\P_{\G}}\left(\alpha,\beta,U^{E}, d\Omega\right)< \infty\:\;\text{and}
\end{equation}
\vspace{-.4cm}
\begin{equation}\label{eq:B2}
\tag{B2}\int_{\P_{\G}}\omega_{\P_{\G}}\left(\alpha,\beta,U^{E}, d\Omega\right)/H_{\G}\left(\alpha, \beta, U^{E}\right)\:\; \text{is functionally independent of $U^{E}$}.
 \end{equation}
The space of multi-shape parameter for the family of Type I Wisharts, i.e., the set of pairs $\left(\alpha,\beta\right)$ that satisfy both conditions $\left(A1\right)$ and $\left(A2\right)$, is denoted by $\mathcal{A}$. Likewise, the space of multi-shape parameter for the family of Type II Wisharts, i.e., the set of pairs $\left(\alpha,\beta\right)$ that satisfy conditions $\left(B1\right)$  and $\left(B2\right)$, is denoted by $\mathcal{B}$.
  
  %
\subsection{The LM  conjecture for identifying $\mathcal{A}$ and $\mathcal{B}$}\label{subsec:H}
After defining Type I \& II Wishart distributions, an important goal of Letac $\&$ Massam in \cite{L07} is to identify $\mathcal{A}$ and $\mathcal{B}$, the associated spaces of multi-shape parameters. When the underlying graph $\G$ is homogeneous both $\mathcal{A}$  and  $\mathcal{B}$ are completely identified in \cite{L07}, but when $\G$ is no longer homogeneous, these spaces are only partially identified. More precisely, Letac \& Massam \cite{L07} identify a subset of $\mathcal{A}$ and  a subset of $\mathcal{B}$ as follows.\\
\indent
Let $\mathcal{P}=\left(C_{1},\cdots,C_{r}\right)$ be a given perfect order of the cliques of $\G$ and $\left(S_{2}, \cdots, S_{r}\right)$ the corresponding
sequence of  separators. For each  separator $S\in\mathscr{S}_{\G}$ let $ J\left(\mathcal{P}, S\right):=\left\{j: S_{j}=S\right\}
$. A set associated with $\mathcal{P}$ and $\mathcal{A}$, denoted by $A_{\mathcal{P}}$, is the set of $\left(\alpha,\beta\right)\in \R^{r}\times
\R^{r-1}$ such that: 
\begin{itemize}
\renewcommand{\labelitemi}{a)}
\item $\sum_{j\in J\left(\mathcal{P}, S\right)}\alpha_{j}-\nu\left(s\right)\beta\left(S\right)=0$,\: for each  $S\neq S_{2}$, where $\nu(S)$, as before,  denotes the multiplicity of the separator $S$;
\renewcommand{\labelitemi}{b)}
\item $\alpha_{j}-\left(c_{j}-1\right)/2>0$,\: for each  $j=2,\ldots,r$;
\renewcommand{\labelitemi}{c)}
\item $\alpha_{1}-\delta_{2}>\left(s_{2}-1\right)/2$,\:  where  $\delta_{2}:=\sum_{j\in J\left(\mathcal{P}, S_{2}\right)}\alpha_{j}-\nu\left(S_{2}\right)\beta_{2}$.
\end{itemize}
Similarly, a set associated with $\mathcal{P}$ and $\mathcal{B}$, denoted by  $B_{\mathcal{P}}$, is the set of $\left(\alpha, \beta\right)$ such that:
\begin{itemize}
\renewcommand{\labelitemi}{a)}
\item $\sum_{j\in J\left(\mathcal{P}, S\right)}\left(\alpha_{j}+\left(c_{j}-s_{j}\right)/2 \right)-\nu\left(S\right)\beta\left(S\right)=0$,\: for each  $S\neq S_{2}$;
\renewcommand{\labelitemi}{b)}
\item $-\alpha_{j}-\left(c_{j}-s_{j}-1\right)/2>0$,\:  for each  $j=2,\ldots,r$  and  $-\alpha_{1}-\left(c_{1}-s_{2}-1\right)/2>0$ ;
\renewcommand{\labelitemi}{c)}
\item $-\alpha_{1}-\left(c_{1}-s_{2}+1\right)/2-\eta_{2}>\left(s_{2}-1\right)/2$ where $\eta_{2}:=\sum_{j\in J\left(\mathcal{P},
S_{2}\right)}\left(\alpha_{j}+\left(c_{j}-s_2\right)/2\right)-\nu\left(S_{2}\right)\beta_{2}$.
\end{itemize}
Theorems 3.3 \& 3.4 in \cite{L07} prove that if $\G$ is a non-complete decomposable graph, then $A_{\mathcal{P}}\subseteq \mathcal{A}$ and $B_{\mathcal{P}}\subseteq \mathcal{B}$. Therefore, $\bigcup_{\mathcal{P}}A_{\mathcal{P}}\subseteq \mathcal{A}$ and   $\bigcup_{\mathcal{P}}B_{\mathcal{P}}\subseteq \mathcal{B}$, where the subscript $\mathcal{P}$ runs through all perfect orders of the cliques of $\G$. When $\G$ is homogeneous Letac and Massam in \cite{L07} establish that $\bigcup_{\mathcal{P}}A_{\mathcal{P}}\subsetneqq \mathcal{A}$ and $\bigcup_{\mathcal{P}}B_{\mathcal{P}}\subsetneqq \mathcal{B}$, but in the case of an arbitrary non-homogeneous decomposable graph they conjecture that equalities hold. We now proceed to formally state the Letac-Massam conjecture.\\

\noindent
\textbf{The Letac-Massam (LM) Conjecture}. Let $\G$ be a non-homogeneous decomposable graph and let $\mathrm{Ord}(\G) $  denote the set of the perfect orders of the cliques of $\G$. Then
\begin{equation}\label{eq:LM1}
\tag{I }\bigcup_{\mathcal{P}\in\mathrm{Ord}(\G) }A_{\mathcal{P}}=\mathcal{A},\qquad\qquad\qquad\qquad\qquad\qquad\qquad\qquad\qquad\qquad\qquad\qquad\qquad\qquad\qquad
\end{equation}
\vspace{-.2cm}
\begin{equation}\label{eq:LM2}
\tag{II} \bigcup_{\mathcal{P}\in \mathrm{Ord}(\G) }B_{\mathcal{P}}=\mathcal{B}.\qquad\qquad\qquad\qquad\qquad\qquad\qquad\qquad\qquad\qquad\qquad\qquad\qquad\qquad\qquad\quad
\end{equation}
\begin{Rem}
Note that for each perfect order $\mathcal{P}=(C_1,\ldots,C_r)$ of the cliques of a decomposable graph $\G$ the sets $A_{\mathcal{P}}$ and $B_{\mathcal{P}}$, as manifolds, are of dimension $r+1$. Therefore, the LM conjecture asserts that $\mathcal{A}$ and $\mathcal{B}$ are also of dimension $r+1$.
\end{Rem}

 \section{The DAG Wishart distributions for directed Markov random fields}\label{sec:DAG Wisharts}

One of the main goals of this paper is to study the LM conjecture and formally  demonstrate that it does not hold in general. Our goal is slightly broader as we are also interested in understanding when exactly the LM conjecture does not hold. In particular, we aim to identify graph characteristics which lead to a violation of the LM conjecture. Our approach is to develop tools which will allow us to compare the Type I \& II Wishart distributions, respectively, with the generalized versions of Riesz distributions, by Andersson et al. \cite{A10} for perfect DAGs, and the DAG Wishart  distributions introduced by the present authors in \cite{BR11}. We demonstrate that relating the LM conjecture to the class of DAGs (and not just undirected graphical models) can provide valuable insights. Since we are able to completely characterize the domain of integrability of the Wisharts associated with DAG models. We begin with a compact review of the  DAG Wisharts given in \cite{BR11}. 
 
\subsection{Gaussian DAG models}\label{subsec:gaussian_over_D} Inference for Gaussian DAG models provide the main motivation for developing the DAG Wisharts in \cite{BR11}. We give a brief introduction here. Let $\D=\left(V, F\right)$ be a DAG with
$p$ vertices, i.e., $|V|=p$. For each $i,j\in V$ let the relation $j\preceq_{\D} i$ denote $i=j$ or $i\mapsto j$. The relation $\preceq_{\D}$ clearly defines a partial order on $V$. Since every partial order can be extended to a linear order  \cite{S30}, without loss of generality, we can assume that the vertices in $V$ are labeled $1, 2,\ldots, p$, and for each $i,j\in V$  if $i\rightarrow
j$, then $i>j$. This order corresponds to the parent order of the vertices of the DAG. Now let the random vector $\mathbf{X}=\left(X_{1},\ldots, X_{p}\right)^{\top}\in \R^{p}$ be a directed Markov random field (or DAG) over
$\D$. Thus $\mathbf{X}$ obeys the directed local Markov property with respect to $\D$, i.e.,
 \begin{equation}\label{eq:DLM}
 j\ind \mathrm{nd}\left(j\right)\setminus\mathrm{pa\left(j\right)}|\mathrm{pa}\left(j\right)\quad \forall j\in V.
 \end{equation}
If, in addition, $\mathbf{X}\sim \mathcal{N}_{p}\left(0,\Sigma\right)$, then a simple observation in \cite{A98} shows that the directed local Markov property in
Equation \eqref{eq:DLM} is satisfied if and only if $\Sigma\succ 0$ and
 \begin{equation}\label{eq:Requirement}
 \Sigma_{\nprec j]}=\Sigma_{\nprec j\succ}\left(\Sigma_{\prec j\succ}\right)^{-1}\Sigma_{\prec j]}\quad\forall j\in V,
 \end{equation}
where ${\nprec j]}:=\left\{\left(i,j\right): i : i\in \mathrm{nd}\left(j\right), i>j\right\}$,
$\nprec j\succ :=\left\{\left(k,i\right):  k\in \mathrm{nd}\left(j\right)\;\; k>j, \:\text{and}\: i\in \mathrm{pa}\left(j\right)\right\}$, $\prec j\succ :=\mathrm{pa}\left(j\right)$ and $\prec j]:=\mathrm{pa}(j)\times \{j\}$.\\
\noindent We define the Gaussian DAG model\footnote{Also said to be Gaussian Bayesian network.}, denoted by $\mathscr{N}\left(\D\right)$, to be the family of all centered Gaussian distribution $\mathcal{N}_{p}\left(0,\Sigma\right)$ which are directed Markov random fields over $\D$. It is easily seen that the distributions in
 $\mathscr{N}\left(\D\right)$ can be parameterized by the space of covariance matrices 
 \begin{equation}\label{eq:PDD}
\mathrm{PD}_{\D}:=\left\{\Sigma\succ 0: \Sigma_{\nprec j]}=\Sigma_{\nprec j\succ}\left(\Sigma_{\prec j\succ}\right)^{-1}\Sigma_{\prec j]}, \:\forall
j\in V\right\}.
 \end{equation}
These distributions can also be parametrized by the space of inverse-covariance matrices $\P_{\D}:=\left\{\Omega:
\Omega^{-1}\in \mathrm{PD}_{\D}\right\}$. Other important parameterizations of the distributions in $\mathscr{N}\left(\D\right)$ are available in terms of the modified Cholesky decompositions of the inverse-covariance matrices, i.e., $\Sigma^{-1}=L \Lambda L^{\top}$ such that $L$ is a lower triangular matrix with all diagonals equal to $1$, and $\Lambda$  is a diagonal matrix. Note that for two distinct vertices $i,j$, if $i$ is not a parent of $j$, then $L_{ij}=0$. We refer the reader to \cite{A98, BR11, BR12} for more details. 
 
\subsection{The DAG Wishart distribution for perfect DAGs }\label{sub:dag_wishart}
 Let $\D$ be a perfect DAG. First note that a random vector $\mathbf{X}$ in $\R^{p}$ is a DAG over $\D$ if and only if it is an undirected graphical model over $\D^{\mathrm{u}}$, the undirected version of $\D$ (which is also necessarily decomposable) \cite{W80}. This implies that $\P_{\D}$ and $\mathrm{PD}_{\D}$ are, respectively, identical to $\P_{\D^{\mathrm{u}}}$ and
$\mathrm{PD}_{\D^{\mathrm{u}}}$ (see \S \ref{sub:gaussian_over_G} for definitions). In particular, $\P_{\D}$ is an open convex cone. The DAG Wishart distribution $\pi_{\P_{\D}}$, as we shall define here is a distribution on $\P_{\D}$ \cite{BR11}. We first define, $\widehat{\pi}_{\P_{\D}}$, the non-normalized version of $\pi_{\P_{\D}}$ as follows:
 \begin{equation}\label{eq:unnormalized_piD}
\widehat{\pi}_{\P_{\D}}\left(\eta, U, d\Omega\right):= \exp\left\{- \frac{1}{2} \tr\left( \Omega U\right) \right\}
\prod_{i=1}^{p}D_{jj}^{-\frac{1}{2}\eta_{j}+pa_{j}+2}d\Omega, 
 \end{equation}
where the multi-shape parameter $\eta$ lives in $\R^{p}$, $U\succ 0$, $D_{jj}:=\Sigma_{jj}-\Sigma_{\prec j]}^{\top}\left(\Sigma_{\prec j
\succ}\right)^{-1}\Sigma_{\prec j]}$ and $pa_{j}=|\mathrm{pa}\left(j\right)|$. From Theorem 4.1 in \cite{BR11} the domain of integrability of the DAG Wishart distribution can be fully characterized:
 \begin{equation*}
 \int_{\P_{\D}}\widehat{\pi}_{\P_{\D}}\left(\eta, U, d\Omega\right)<\infty\Longleftrightarrow \eta_{j}>pa_{j}+2\quad \forall j\in V.
  \end{equation*}
Moreover, if $\eta_{j}>pa_{j}+2\quad \forall j\in V$, then the normalizing constant is given by
 \begin{equation} \label{eq:normalizing_constant}
z_{\D}\left(U,\eta\right) := \prod_{j=1}^{p} \frac{\Gamma \left( \frac{\eta_{j}}{2} -
\frac{pa_{j}}{2} - 1 \right) 2^{\frac{\eta_{j}}{2} - 1}
\left(\sqrt{\pi}\right)^{pa_{j}} \det\left( U_{\prec j\succ}\right)^{\frac{\eta_{j}}{2} -
\frac{pa_{j}}{2} - \frac{3}{2}}}{\det\left( U_{\preceq j\succeq}\right)^{\frac{\eta_{j}}{2} - \frac{pa_{j}}{2} - 1}},
\end{equation} 
 where   $\preceq j\succeq:=\mathrm{pa}(j)\cup \left\{j\right\}$.
 \begin{Rem}\label{rem:cholesky_decomposition}
Let $\Omega\in \P_{\D}$ and $\Omega=L\Lambda L^{\top}$ be the modified Cholseky decompositions of $\Omega$. Then $L$ is a lower
triangular matrix with all diagonal entries equal to one and $L_{ij}=0\quad \forall \left(i,j\right)\notin F, $ and $\Lambda$ is a diagonal
matrix such that $\Lambda_{jj}=\left(\Sigma_{jj| \prec j\succ}\right)^{-1}=D_{jj}^{-1}$  \cite{W80}.
 \end{Rem}
 
\section{Comparing decomposable Type II Wisharts with perfect DAG Wisharts} \label{sub:comparison_non_hom}
We now proceed to compare Type II Wisharts for decomposable graphs with DAG Wisharts for perfect DAGs. We had noted earlier that the class of decomposable graphs are Markov equivalent to the class of perfect DAGs. In particular, every probability distribution that obeys the global Markov property with respect to a decomposable graph  also obeys the local directed Markov property with respect to a perfect DAG version and vice versa. In the Gaussian setting this means that if $\D$ is a perfect DAG version of $\G$, then $\mathcal{N}\left(\G\right)$ and $\mathcal{N}\left(\D\right)$ define the same family of $p$-variate Gaussian distributions. Consequently,  the family of Type II Wisharts and the family of DAG Wisharts are both defined on $\P_{\G}=\P_{\D}$, the space of inverse-covariance matrices.  Therefore, a relevant question is how the functional form of the Type II Wishart density compares with that of the DAG Wishart. \\
 \indent
First, to facilitate comparison, we re-parameterize the DAG Wishart $\pi_{\P_{\D}}$  as follows. For each $j=1,\ldots, p$ let the expressions of the form  $-\frac{1}{2}\eta_{j}+pa_{j}+2$   in Equation \eqref{eq:unnormalized_piD} be replaced by $\gamma_{j}$, and let $U$ be replaced by $2U$. Let $\gamma:=\left(\gamma_{1},\ldots,\gamma_{p}\right)$. Under this parametrization, with a slight abuse of notation, we write 
 \[
 \pi_{\P_{\D}}\left(\gamma, U, d\Omega\right)=z_{\D}\left(U, \gamma\right)^{-1}\exp\left\{-\tr\left( \Omega U\right) \right\}
\prod_{i=1}^{p}D_{jj}^{\gamma_{j}}\;d\Omega\,,
  \]
where  the normalizing constant  $z_{\D}\left(\gamma\right)$ exists if and only if   $\gamma_{j}<pa_{j}/2+1$ for each  $j=1,\ldots,p$  and 
\begin{equation}\label{eq:normalizing_constant_rewritten}
z_{\D}\left(U,\gamma\right) = \prod_{j=1}^{p} \frac{\Gamma \left( -\gamma_{j}+\frac{pa_{j}}{2}+1 \right) 
\left(\sqrt{\pi}\right)^{pa_{j}}}{ U_{jj|\prec j\succ}^{-\gamma_{j} +\frac{pa_{j}}{2} +1}\det\left( U_{\prec j\succ}\right)^{\frac{1}{2}}}.
\end{equation}
Note also that for each perfect directed version $\D$ of $\G$ the exponential term $\exp\left\{-\tr \left( \Omega^{-1} U \right)\right\}$ is a common term in both the Type II Letac-Massam Wishart $\W_{\P_{\G}}$ and the DAG Wishart $\pi_{\P_\D}$. Moreover, the DAG Wishart $\pi_{\P_\D}$ has additional terms only of the form $D_{jj}^{\gamma_j}$. Before comparing $\W_{\P_{\G}}$ and $\pi_{\P_{\D}}$ more generally, we first illustrate the comparison with an example. 
 \begin{figure}[ht]
\centering
\subfigure[]{
\includegraphics[width=3.3cm]{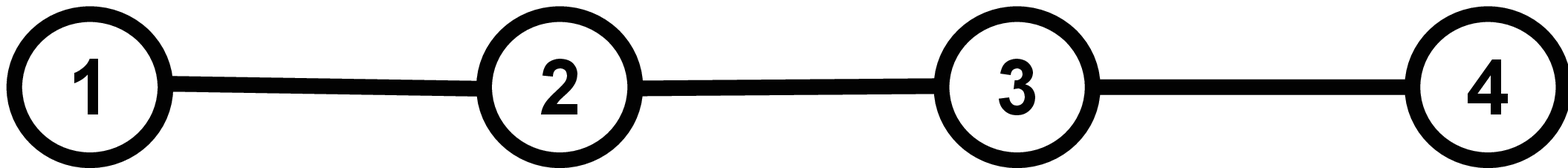}
\label{fig:graph-patha}
}
\hspace{2cm}
\subfigure[]{
\includegraphics[width=3.3cm]{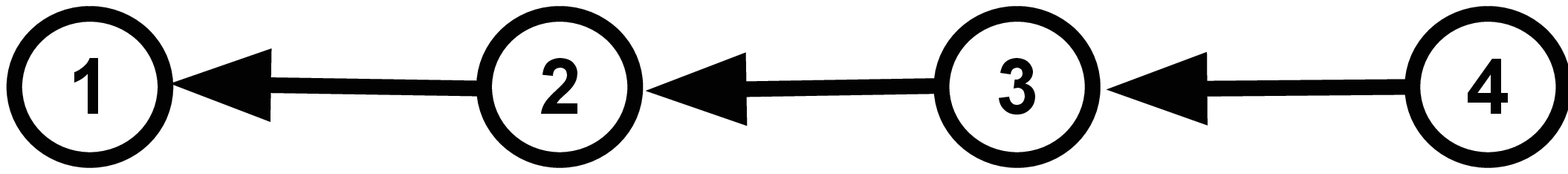}
\label{fig:graph-pathb}
}
\hspace{2cm}
\subfigure[]{
\includegraphics[width=3.3cm]{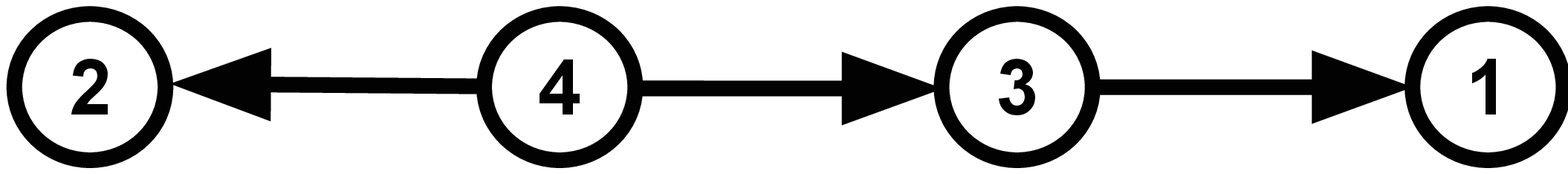}
\label{fig:graph-pathc}
}

\caption{ A $4$-path $A_4$ with  two directed versions of it. }
\end{figure}

\begin{Ex}\label{ex:comparison} Let $\G$ be the $4$-path given in Figure \ref{fig:graph-patha}. Note that $\G$ is a non-homogeneous decomposable graph. It is clear that the DAG given in Figure \ref{fig:graph-pathb} is a perfect DAG version of $\G$. The cliques of $\G$ are  $C_{1} = \left\{1,2\right\}, C_{2} =
\left\{2,3\right\}, C_{3} = \left\{3,4\right\}$ and the separators are  $S_{2}= \left\{2\right\}, S_{3} = \left\{3\right\}$.  Let $\Sigma=\Omega^{-1}$. To compare the corresponding Letac-Massam Type II Wishart and the DAG Wishart for this graph, we rewrite the Markov ratio present in the  density of  $\W_{P_G}$ as follows. 
\begin{eqnarray}\label{eq:comparison_A4}
\notag\frac{\prod_{i=1}^3 \det\left( \Sigma_{{C}_{j}} \right)^{\alpha_{j}+\frac{c_{j}+1}{2}}}
{\prod_{j=2}^3 \det\left( \Sigma_{S_{j}} \right)^{\beta_{j}+\frac{s_{j}+1}{2}}}&=&\frac{\det\left( \Sigma_{\preceq 1\succeq} \right)^{\alpha_{1}+\frac{3}{2}}\det\left(
\Sigma_{\preceq 2\succeq} \right)^{\alpha_{2}+\frac{3}{2}}\det\left( \Sigma_{\preceq 3\succeq} \right)^{\alpha_{j}+\frac{3}{2}}}
{\det\left( \Sigma_{\prec 1\succ} \right)^{\beta_{2}+1}\det\left( \Sigma_{\prec 2\succ} \right)^{\beta_{3}+1}}\\
&=&D_{11}^{\alpha_{3}+\frac{3}{2}}D_{22}^{\alpha_{1}+\frac{3}{2}}D_{33}^{\alpha_{2}+\frac{3}{2}}D_{44}^{\alpha_{2}+\frac{3}{2}}\Sigma_{22}^{\alpha_{1}-\beta_{2}+\frac{1}{2}}\Sigma_{33}^{\alpha_{3}-\beta_3+\frac{1}{2}}
\end{eqnarray}
As shown in section 3.4 of \cite{L07}, one can check that $\mathcal{B}=B_{\mathcal{P}_{1}}\cup B_{\mathcal{P}_{2}}$ where $\mathcal{P}_{1}=\left(C_{1}, C_{2}, C_{3}\right)$ and
$\mathcal{P}_{2}=\left(C_{2}, C_{1}, C_{3}\right)$ are the perfect orders of the cliques of $\G$ and
\begin{align*}
B_{\mathcal{P}_{1}}&=\left\{\left(\alpha_{1},\alpha_{2},\alpha_{3},\beta_{2},\beta_{3}\right): \alpha_{i}<0, -\alpha_{1}-\alpha_{2}+\beta_{2}-1>0,
\alpha_{3}-\beta_{3}+\frac{1}{2}=0\right\},\\
B_{\mathcal{P}_{2}}&=\left\{\left(\alpha_{1},\alpha_{2},\alpha_{3},\beta_{2},\beta_{3}\right): \alpha_{i}<0, -\alpha_{2}-\alpha_{3}+\beta_{3}-1>0,
\alpha_{1}-\beta_{2}+\frac{1}{2}=0\right\}.
\end{align*}
 Note that unless $\alpha_{3}-\beta_{3}+\frac{1}{2}=0$ and $\alpha_{1}-\beta_{2}+\frac{1}{2}=0$ ( i.e., $\left(\alpha, \beta\right)$ is restricted to the intersection of $B_{\mathcal{P}_{1}}$ and $B_{\mathcal{P}_{2}}$) the expression in Equation \eqref{eq:comparison_A4} contains some terms different from the product of $D_{jj}$ to some powers. Since the DAG Wishart has polynomial terms only of the form $D_{jj}^{\gamma_j}$, it is clear that for this directed version of $\G$,  $\W_{\P_{\G}}$ and $\pi_{\P_{\D}}$ are not directly comparable. Note that we did not need to account for two additional perfect orders $P_1^{\prime} = (C_3 , C_2 , C_1 )$ and $P_2^{\prime} =(C_2 , C_3 , C_1 )$ since it has been shown in \cite{L07} that $\mathcal{B}_{P_1^{\prime}}= \mathcal{B}_{P_1}$ and $\mathcal{B}_{P_2^{\prime}}=\mathcal{B}_{P_2}$. Now consider the comparison with the perfect DAG version given in Figure \ref{fig:graph-pathc}. In this case the cliques are $C_{1}=\left\{2,4\right\}, C_{2}=\left\{3,4\right\}, C_{3}=\left\{1,3\right\}$ and the separators are $S_{2}=\left\{ 4\right\}$ , $S_{3}=\left\{3\right\}$. By a
similar calculation as that in Equation \eqref{eq:comparison_A4} we  obtain
\begin{equation*}
\frac{\prod_{i=1}^3 \det\left( \Sigma_{{C}_{j}} \right)^{\alpha_{j}+\frac{c_{j}+1}{2}}}
{\prod_{j=2}^3 \det\left( \Sigma_{S_{j}}
\right)^{\beta_{j}+\frac{s_{j}+1}{2}}}=D_{11}^{\alpha_{1}+\frac{3}{2}}D_{22}^{\alpha_{3}+\frac{3}{2}}D_{33}^{\alpha_{2}+\frac{3}{2}}D_{44}^{\alpha_{1}+\alpha_{2}-\beta_{2}+2}.
\end{equation*}
Therefore, for this directed version of $\G$, the family of Type II Wisharts, restricted to $\mathcal{B}_{\mathcal{P}_{1}}$, is a subfamily of the family of DAG Wisharts.
\end{Ex}
\noindent
In Example \ref{ex:comparison} we illustrated the fact that although $\W_{\P_{\G}}$ does not, necessarily, compare with  $\pi_{\P_{\D}}$ for any arbitrary perfect DAG version  $\D$ of $\G$, it is however comparable with some particular DAGs. We will show next that this conclusion can be generalized to any decomposable graphs. To this end, we proceed with a few useful lemmas. In particular, we introduce tools that will allow us to relate decomposable graphs with a given perfect ordering of its cliques with perfect DAGs and vice versa. These tools turn out to be critical ingredients in comparing the  undirected Letac-Massam Wisharts to the directed DAG Wisharts. Before we proceed with the next lemma we introduce some convenient notation and a definition.
\begin{Notation}\label{not:->}
Let $\D=(V,F)$ be a DAG and let $A,B\subset V$. Then $A\to B$ denotes the fact that there exist a vertex $v\in A\setminus B$ and a vertex $v'\in B$ such that $v\to v'$.
\end{Notation}
\begin{definition}
Let $\mathcal{P}=\left(C_{1},\ldots, C_{r}\right)$ be a perfect order of the cliques of a decomposable graph $\G$. A DAG version $\D$ of $\G$  is said to be induced by $\mathcal{P}$ if $H_{1}, \dots, H_{r-1}$ are all ancestral in $\D$. 
\end{definition}
\begin{Rem}
Note that in terms of Notation \ref{not:->} a DAG version $\D$ of $\G$ is induced by $\mathcal{P}=\left(C_{1},\ldots, C_{r}\right)$ if and only if $C_i\to C_j$ implies that $i< j$. 
\end{Rem}
	
 \begin{lemma}\label{lem:combine}
Let $\G$ be a non-complete decomposable graph.
\begin{enumerate}
\item[a)] Let  $\mathcal{P}=\left(C_{1}, \ldots, C_{r}\right)$ be a perfect order of the cliques of a decomposable graph $\G$. Then every DAG version of $\G$ induced by $\mathcal{P}$ is a perfect DAG. Moreover , there exists a perfect DAG  version $\D$ of $\G$ induced by $\mathcal{P}$ such that $S_{2}$ is ancestral in $\D$. 
\item[b)] Conversely, suppose $\D$ is a perfect DAG version of $\G$. Then there exists a perfect order $\mathcal{P}$, of the cliques of $\G$, such that $\D$ is induced by $\mathcal{P}$.
\end{enumerate}
 \end{lemma}
 \begin{proof}
a)~ Suppose, to the contrary, that $\D$ is not perfect. Let  $j$ be the smallest integer such that $D_{H_{j}}$, the induced DAG on $H_{j}$, is not perfect. It is clear that $1<j\leq r$.  Let $v\rightarrow v'\leftarrow v''$ be an immorality in $\D_{H_{j}}$. This, in particular, implies that there are two distinct cliques $C_{j_{1}}$ and $C_{j_{2}}$, with subscript $j_{1}, j_{2}\leq j$, such that they contain $v,v'$ and $v', v''$, respectively.  Since $j_{1}$ and $j_{2}$ are distinct we may assume that $j_{1}<j$. But since $H_{j-1}$ is ancestral and  $v''$ is a parent of $v'\in H_{j-1}$ we must have $v''\in H_{j-1}$. This contradicts the fact that the induced DAG on $H_{j-1}$ is perfect. \\
\indent Now we show that in particular there exists a DAG $\D$ induced by $\mathcal{P}$ such that $S_2$ is ancestral in $\D$. First consider the case where there are only two cliques. We start with relabeling the vertices in $S_{2}$, $H_{1}\setminus S_{2}$ and $R_{2}$, respectively, in a decreasing order. If $\D$ is the DAG version of $\G$ induced by this order, then $S_{2}$ and  $H_{1}$ are ancestral in $\D$. Now suppose that such a DAG version exists for any decomposable graph with number of cliques less than $r \geq 3$. By the mathematical induction there exists a DAG version $\D'$ of $\G_{H_{r-1}}$ such that $S_{2}, H_{1}, \ldots, H_{r-2}$ are ancestral in $\D'$. Without loss of generality, we can assume that the vertices in $\D'$ are labeled from $p, \ldots, p-|R_{r}|$. Let us label the vertices in $R_{r}$ from $1,\ldots, |R_{r}|$ and let $\D$ be the DAG version of $\G$ induced by this order. One can easily check that $\D$ has the desired properties.\\
\indent b)~ Suppose that by mathematical induction the lemma holds for any DAG with fewer than $p$ vertices. Now we assume that $\G$ is a decomposable graph with $p$ vertices. Clearly, we can assume that $p\ge 2$. Let $\G^{\prime}$ be the induced graph on $V\setminus\{1 \}$. By our induction hypothesis, there is a perfect order $\mathcal{P}^{\prime}=(C_1,\ldots, C_k)$ of $\G^{\prime}$ such that $\D^{\prime}$, the induced DAG on $V\setminus\{ 1\}$, is induced by $\mathcal{P}^{\prime}$. Note that $C:=\preceq 1\succeq=$ is a clique of $\G$. Consider two possible cases:
\begin{enumerate}
\item[a)]  There is an $i$ such that $C_i$ is not a clique in $\G$. Then $C=C_i\cup \{1\}$. This implies that for each $j\neq i$, $C_j$ remains a clique in $\G$. Let us replace $C_i$ with $C$ and define 
\[
\mathcal{P}=(C_1,\ldots,\underbrace{C}_{i},\ldots, C_k).
\]
 One can easily check that $\mathcal{P}$ is a perfect order of $\G$, and $\D$ is induced by $\mathcal{P}$. 
\item[b)] For every $i=1,\ldots,k$, $C_i$ is a clique in $\G$. Let
\[
i_1:=\max\{i:\: C_i\to C\} \:\; \text{and} \:\;  i_2:=\min\{i:\: C \to C_i \}.
\] 
We use the convention that $\max\emptyset=-\infty$ and $\min\emptyset=+\infty$. First suppose $i_1,i_2$ are both finite. Thus we have $C_{i_1}\to C\to C_{i_2}$ and therefore $i_1< i_2$, because by our induction hypothesis the histories of $\mathcal{P}'$ are ancestral in $\D'$. One can check that $\mathcal{P}=(C_1,\ldots,C_{i_1},C,\ldots,C_k)$ is then a perfect order of $\G$ and $\D$ is induced by $\mathcal{P}$. If $i_1=+\infty$ or $i_2=-\infty$, then by appending the clique $C$  at the end or at the beginning of $(C_1,\ldots, C_k)$, respectively, we obtain a perfect order $\mathcal{P}$ and in either case $\D$ is induced by such $\mathcal{P}$.
\end{enumerate}
\end{proof}
\begin{lemma}\label{lem:ancestral}
Let $\D$ be a DAG and let $\mathbf{X}\sim \mathcal{N}_{p}\left(0, \Sigma\right)\in \mathscr{N}\left(\D\right)$.  Suppose  $\Sigma^{-1}=\Omega=LD^{-1}L^{\top}$ is the modified  Cholesky decomposition of $\Sigma^{-1}$. Then we have:
\begin{itemize}
\renewcommand{\labelitemi}{i)}
\item For each $i,j\in V$ if $i\in \mathrm{pa}\left(j\right)$, then $L_{ij}=-\beta_{ji}$, where $\beta_{ji}$ is the partial regression coefficient  of $X_{i}$ in the linear regression of $X_{j}$ on $\mathbf{X}_{\prec i\succ}$  and $D_{jj}=\Sigma_{jj|\prec j\succ}$.
\renewcommand{\labelitemi}{ii)}
\item If $A$ is an ancestral subset of $V$, then $\left(\Sigma_{A}\right)^{-1}=L_{A}D_{A}^{-1}L_{A}^{\top}$. In particular, $\det\left(\Sigma_{A}\right)=\prod_{j\in A}D_{jj}$ (also see \cite{KR12} for a related result).
\end{itemize}
\end{lemma}
\begin{proof}\
\begin{itemize}
\renewcommand{\labelitemi}{i)}
\item This can be proved by using Equation \eqref{eq:Requirement} (see  \cite{BR11, W80} for details).
\renewcommand{\labelitemi}{ii)}
\item Since $A$ is ancestral in $\D$, by using Equation \eqref{eq:PDD}, one can easily show that $\mathbf{X}_{A}\sim\mathcal{N}_{|A|}\left(0, \Sigma_{A}\right)\in \mathscr{N}\left(\D_{A}\right)$. Now let $\left(\Sigma_{A}\right)^{-1}=KF^{-1}K^{\top}$ be the modified Cholesky decomposition of $\left(\Sigma_{A}\right)^{-1}$. Part i) and the fact that $A$ is ancestral imply that $K_{ij}=-\beta_{ij}=L_{ij}$ whenever $i\in \mathrm{pa}\left(j\right)$, and $F_{jj}=\Sigma_{jj|\prec j\succ}=D_{jj}$. This implies that $K=L_{A}$ and $F=D_{A}$ 
\end{itemize}
\end{proof}

 \begin{Theorem}\label{thm:comparison}
For every perfect order $\mathcal{P}=\left(C_{1},\ldots, C_{r}\right)$ of the cliques of a decomposable $\G$ there exists a perfect DAG version $\D$
of $\G$ such that the family of distributions $\W_{\P_{\G}}$, restricted to $\mathcal{B}_{\mathcal{P}}$, is a subfamily of the family of $ distributions \; \pi_{\P_{\D}}$.
\end{Theorem}
\begin{proof}
Let $\D$ be a DAG version of $\G$  induced by $\mathcal{P}$. First note that by Lemma  \ref{lem:combine} $\D$ is perfect. Therefore $\P_{D}$ is identical to $\P_{\G}$ and $\mathrm{\pi}_{\P_{\D}}$ is indeed a distribution on $\P_{\G}$. In comparing $\W_{\P_{\G}}$ with $\pi_{\P_{\D}}$ it suffices to show that the Markov ratio that appears in the density of $\W_{\P_{\G}}$ can be written as products of $D_{jj}$  to some powers. We proceed to rewrite the corresponding  Markov ratio as follows:
\begin{align}\label{eq:markov_ratio_rewritten}
\notag&\frac{\prod_{j=1}^{r} \det\left( \Sigma_{C_{j}} \right)^{\alpha_{j}+\frac{c_{j}+1}{2}}}
{\prod_{j=2}^{r} \det\left( \Sigma_{S_{j}} \right)^{\beta_{j}+\frac{s_{j}+1}{2}}}=\det\left(\Sigma_{R_{1}|S_{2}}\right)^{\alpha_{1}+\frac{c_{1}+1}{2}}\det\left(\Sigma_{S_{2}}\right)^{\alpha_{1}+\frac{c_{1}+1}{2}}\prod_{j=2}^{r}\det\left(\Sigma_{R_{j}|S_{j}}\right)^{\alpha_{j}+\frac{c_{j}+1}{2}}\\
\notag&\times\prod_{j=2}^{r}\det\left(\Sigma_{S_{j}}\right)^{\alpha_{j}-\beta_{j}+\frac{c_{j}-s_{j}}{2}}\\
\notag &=\det\left(\Sigma_{R_{1}|S_{2}}\right)^{\alpha_{1}+\frac{c_{1}+1}{2}}\prod_{j=3}\det\left(\Sigma_{R_{j}|S_{j}}\right)^{\alpha_{j}+\frac{c_{j}+1}{2}}\det\left(\Sigma_{S_{2}}\right)^{\alpha_{1}+\frac{c_{1}+1}{2}+\sum_{j\in J\left(\mathcal{P}, S_{2}\right) }
\left(\alpha_{j}+\frac{c_{j}-s_{2}}{2} \right)-\nu\left(S_{2}\right)\beta\left(S_{2}\right)}\\
&\times\prod_{S\in \mathscr{S}_{\G}\setminus{S_{2}}}^{r}\det\left(\Sigma_{S}\right)^{\sum_{j\in
J\left(\mathcal{P}, S\right)} \left(\alpha_{j}+\frac{c_j-|S|}{2} \right)-\nu\left(S\right)\beta\left(S\right)}.
\end{align}
Let $K_{j}:=H_{j}\setminus C_{j}$ for each $j=2,\ldots, r$. Consider the following block-partitioning of $\Sigma_{H_{j}}$.
\[
\Sigma_{H_{j}}=
\left(
\begin{matrix}
\Sigma_{R_{j}}& \Sigma_{R_{j}S_{j}}&\Sigma_{R_{j}K_{j}}\\
\Sigma_{S_{j}R_{j}}&\Sigma_{S_{j}}&\Sigma_{S_{j}K_{j}}\\
\Sigma_{K_{j}R_{j}}&\Sigma_{K_{j}S_{j}}&\Sigma_{K_{j}}
\end{matrix}
\right).
\]
Now for each $j=2,\ldots,r$ we have $\Sigma_{H_{j}}\in \mathrm{PD}_{\G_{H_{j}}}$, and $S_{j}$ separates $R_{j}$  form $K_{j}$ . By Lemma 5.5 \cite{L96} we have 
\[
\det\left(\Sigma_{H_{j}}\right)=\dfrac{\det\left(\Sigma_{C_{j}}\right)\det\left(\Sigma_{H_{j-1}}\right)}{\det\left(\Sigma_{S_{j}}\right)}.
\]
 By rewriting this and using Lemma \ref{lem:ancestral} we obtain
\[
\det\left(\Sigma_{R_{j}|S_{j}}\right)=\det\left(\Sigma_{H_{j}}\right)\det\left(\Sigma_{H_{j-1}}\right)^{-1}=\prod_{\ell\in R_{j}}D_{\ell\ell}.
\]

 Similarly, Lemma \ref{lem:ancestral} implies that $\det\left(\Sigma_{S_{2}}\right)=\prod_{\ell\in S_{2}}D_{\ell\ell}$ and  $\det\left(\Sigma_{R_{1}|S_{2}}\right)=\prod_{\ell\in R_{1}}D_{\ell\ell}$. Now note that if $\left(\alpha, \beta\right)\in
B_{\mathcal{P}}$ and $S\neq S_{2}$,  then $\sum_{j\in J\left(\mathcal{P}, S\right)}\left(\alpha_{j}+\left(c_{j}-|S|\right)/2 \right)-\nu\left(s\right)\beta\left(S\right)=0$. Thus when the shape parameters  are  restricted to  $B_{\mathcal{P}}$  the Markov ratio above is only a product of some powers of $D_{jj}$.
\end{proof}
The next proposition is essential for our purposes as it gives us the recipe that we need to construct counterexamples to the LM conjecture \eqref{eq:LM2}. Note that if the LM conjecture (II) is true, then the dimension of the set $\mathcal{B}$, as a manifold, is $r+1$. First we introduce a new notation as follows.
\begin{definition}\label{def:ans}
Let $\D$ be a DAG version of $\G$ induced by $\mathcal{P}$ and let $\mathscr{S}_{\G}^{\D}$ denote the set of all separators $S\in \mathscr{S}_{\G}$ that are ancestral in $\D$.  We denote the number of elements of $\mathscr{S}_{\G}^{\D}$ by $r_{\D}$.
\end{definition}
\begin{proposition}\label{prop:B1}
 Let $\D$ be a DAG version of $\G$ induced by $\mathcal{P}$. Then the dimension of the manifold described by  $\left(\alpha,\beta\right)\in \R^{r}\times \R^{r-1}$ that satisfies Equation \eqref{eq:B1} is greater than or equal to $r+r_{\D}$.
\end{proposition}
 \begin{proof}
 Suppose $S\in \mathscr{S}_{\G}^{\D}$. By part ii) of Lemma \ref{lem:ancestral} the term
  \[
 \det\left(\Sigma_{S}\right)^{\sum_{j\in
 J\left(\mathcal{P}, S\right), \: S\neq S_{2}} \left(\alpha_{j}+\frac{c_j-|S|}{2} \right)-\nu\left(S\right)\beta\left(S\right)},
 \]
 in Equation \eqref{eq:markov_ratio_rewritten} can be written as products of some powers of $D_{jj}$. This in turn implies that for each $S\in \mathscr{S}_{\G}^{\D}$, restricting $\left(\alpha, \beta\right)$ to the equation
 \[
 \sum_{j\in
 J\left(\mathcal{P}, S\right), S\neq S_{2}} \left(\alpha_{j}+\frac{c_j-|S|}{2} \right)-\nu\left(S\right)\beta\left(S\right)=0
 \]
 is not necessary and, consequently, the dimension of the corresponding set, as a manifold, is at least $r+r_{\D}$.
 \end{proof}
   \begin{Rem}\label{rem:key}
 As we mentioned earlier, the LM conjecture (II) implies that for any decomposable graph $\G$ the dimension of $\mathcal{B}$ is $r+1$. Now in light of Proposition \ref{prop:B1} the LM conjecture (II) suggests that for any DAG version $\D$ of $\G$ the number $r_{\D}\le 1$. Therefore, the LM conjecture (II) can be shown not to be true if we can construct a decomposable graph $\G$ and a DAG version $\D$ such that $r_{\D}>1$. We shall further exploit this line of reasoning in \S \ref{sub:LM_partb}. 
\end{Rem}
\section{Comparing Homogeneous Type II Wisharts with Perfect Transitive DAG Wisharts}\label{sub:typeII_dag_wish_hom} Henceforth in this section, let $\mathcal{H}=\left(V, H\right)$ denote a homogeneous graph. In this section we show that for any homogeneous graph $\G$ there is a DAG version $\D$ such that $\W_{\P_{\G}}$ is a special case of $\pi_{\P_{\D}}$ on the whole parameter set $\mathcal{B}$ (note that, as we  discussed in \S \ref{subsec:H}, when a graph is homogeneous the parameter set $\mathcal{B}$ is completely identified). Note also that Equation \eqref{eq:homogeneous} defining homogeneous graphs naturally defines a partial order on the vertex set $V$ of $\mathcal{H}$ as follows:
\[
\forall i,j\in V,\:\; i\succeq_{\mathcal{H}} j \Longleftrightarrow \mathrm{ne}\left(j\right)\cup\left\{ j\right\}\subseteq \mathrm{ne}\left(i\right)\cup\left\{ i\right\}.
\]
See \cite{L07} for more details on rooted Hasse trees and on this partial order $\succeq_{\mathcal{H}}$. We let the linear order $\geq_{\mathcal{H}}$ (or simply $\geq$ when there is no danger of confusion) be a linear extension of the partial order $\succeq_{\mathcal{H}}$, and let $\D$ be the DAG version of $\mathcal{H}$ induced by  $\geq_{\mathcal{H}}$. One can easily check that $\D$ is perfect, and transitive, i.e., 
\[
i\rightarrow j  \; \&\;  j\rightarrow k \implies i\rightarrow k.
\]
The above shows that any homogeneous graph has a perfect transitive DAG version.  Now for a homogeneous graph we prove the following generalized form of Theorem \ref{thm:comparison}.
\begin{proposition}\label{prop:comparison_hom}
Let $\mathcal{H}$ be a homogeneous graph and $\D$ a perfect transitive  DAG version. Then the family of Type II Wisharts for $\mathcal{H}$ is a subfamily of the DAG Wisharts for $\D$.
\end{proposition}
\begin{proof}

First we claim that all the cliques, and consequently all the separators, of $\G$ are ancestral in $\D$. To see this note the following: Let $C$ be a clique of $\G$ and suppose $u\rightarrow v$ for some $v\in C$. Let $w$ be any other vertex in $C$. Then either $v\rightarrow w$ or $v\leftarrow w$. Regardless, since $\D$ is homogeneous, $u$ and $w$  must be adjacent. Thus $u\in C$. This proves that $C$ is ancestral in $\D$. Now if $S$ is a separator of $\G$, then the fact that $S=C\cap C' $ for some $C, C'\in \mathscr{C}_{\G}$ implies that $S$  is ancestral (otherwise, it implies that  the clique $C^{'}$ is not ancestral leading to a contradiction). Therefore by Lemma \ref{lem:ancestral} we obtain
 \[
 \dfrac{\prod_{C\in \mathscr{C}}\det\left( \Sigma_{C} \right)^{\alpha\left(C\right)+\frac{|C|+1}{2}}}
{\prod_{S\in \mathscr{S}}\det\left( \Sigma_{S} \right)^{\beta\left(S\right)+\frac{|S|+1}{2}}}=\frac{\prod_{C\in \mathscr{C}}\prod_{\ell\in C}D_{\ell\ell}^{\alpha\left(C\right)+\frac{|C|+1}{2}}}{\prod_{S\in \mathscr{S}}\prod_{\ell\in S}D_{\ell\ell}^{\beta\left(S\right)+\frac{|S|+1}{2}}}\;,
 \]
 which obviously shows that this Markov ratio is a product of powers of $D_{jj}$, therefore the family of Type II Wisharts for $\mathcal{H}$ is a subfamily of the DAG Wisharts for $\D$.
 \end{proof}

\noindent
 In the following two examples we compare $W_{\P_{\mathcal {G}}}$ and $\pi_{\P_{\D}}$ in more detail. More specifically, we shall explain how the space of shape parameters is identified for each family of distributions. The space of shape parameters for $\W_{\P_{\mathcal{H}}}$ is identified by Theorem 3.2 in \cite{L07}. For this purpose, we follow the notion introduced in \cite[section 3.3]{L07}. \\

For any homogeneous graph $\mathcal{H}$, let $\mathcal{T}_{\mathcal{H}}=\left(T, E_{\mathcal{H}}, \preceq\right)$ be the Hasse tree of $\mathcal{H}$ (see \cite[section 2.2]{L07} for greater detail). Note that a vertex in $T$ is indeed an equivalent class $[i]\subseteq V$ for some vertex $i\in V$, where $j\in [i]$ if and only if $\mathrm{ne}\left(j\right)\cup\left\{j\right\}= \mathrm{ne}\left(i\right)\cup\left\{i\right\}$.  If $t\in T$ is an internal vertex, i.e., a vertex that has a child, then $C_{t}:=\bigcup\left\{[i]:\:  [i]\preceq t \right\}$ is a clique of $\mathcal{H}$. Also if $q \in T$ is a leaf vertex, i.e., has no child, then  $S_{q}=\bigcup\left\{[i]:\:  [i]\preceq q\right\}$ is a separator of $\mathcal{H}$. For each $[i]\in T$  define
\begin{align*}
\rho_{[i]}\left(\alpha,\beta\right)&:=\sum_{[i]\preceq t}\alpha\left(C_{t}\right)-\sum_{[i]\preceq q}\nu\left(S_{q}\right)\beta\left(S_{q}\right),\\
m_{[i]}&:=\sum_{t\preceq [i]}n_{t},
\end{align*}
where $n_{t}$ is the number of the elements in $t$. By Theorem 3.2 in \cite{L07}, 
\[
\left(\alpha, \beta\right)\in \mathcal{B}\Longleftrightarrow  -\rho_{[i]}\left(\alpha,\beta\right)>\left(\sum_{[i]\preceq t}n_{t}-1\right)/2,\: \ \forall [i]\in T.
\]
\begin{figure}[ht]
\centering
\subfigure[]{
\includegraphics[width=3.2cm]{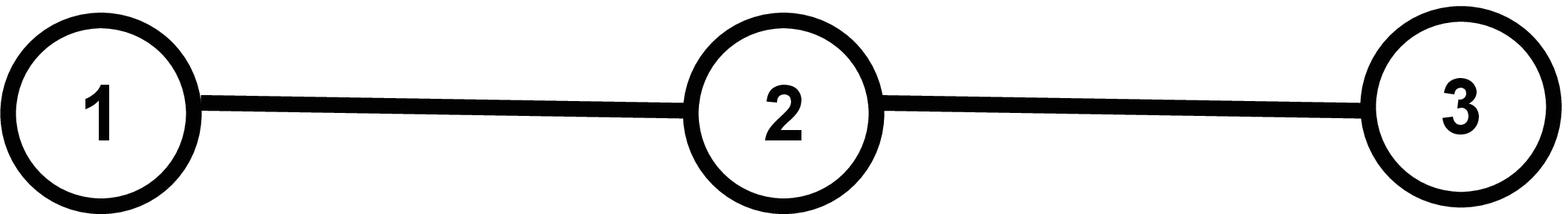}
\label{fig:path-3a}
}
\hspace{2cm}
\subfigure[]{
\includegraphics[width=3.1cm]{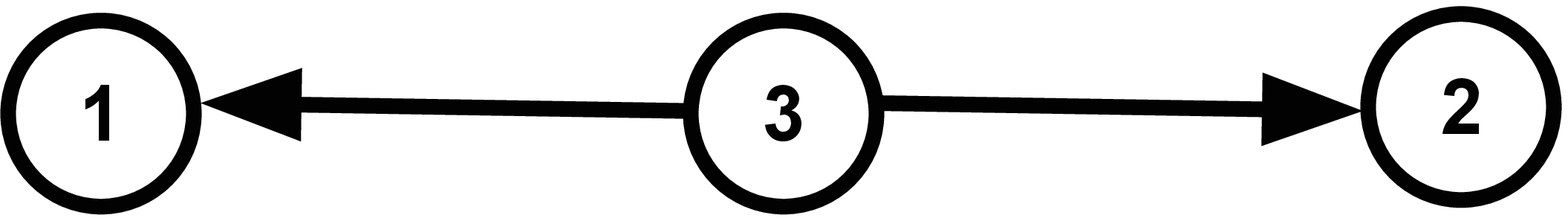}
\label{fig:path-3b}
}
\hspace{2cm}
\subfigure[]{
\includegraphics[width=3.3cm]{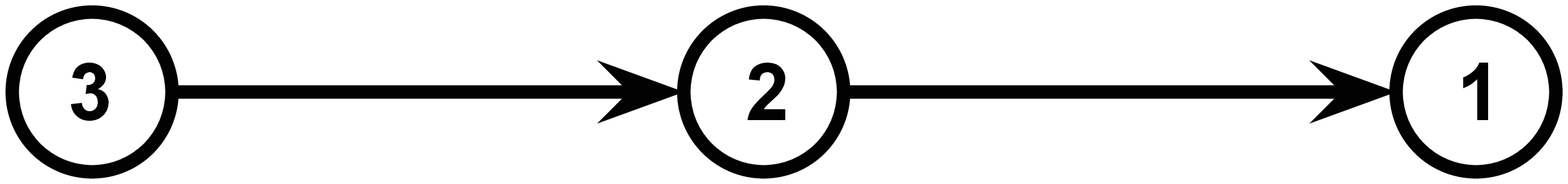}
\label{fig:path-3c}
}
\caption{Denoted graphs are $\left(a\right)$ A homogeneous graph  $A_{3}$ , $\left(b\right)$~  A transitive perfect DAG version of  $A_3$ and $\left(c\right)$~ A perfect DAG version of $A_3$.} 
\end{figure}
\noindent We now proceed to compare the space of shape parameters of $\W_{\P_{\G}}$ and $\pi_{\P_{\D}}$ in two concrete examples.
\begin{Ex}\label{Ex}
Let $\mathcal{H}$ be the $3$-path given in Figure \ref{fig:path-3a} and $\D$ the DAG version given in Figure
\ref{fig:path-3b}. It is clear that $\mathcal{H}$ is a homogeneous graph and $\D$ is a perfect transitive DAG version. First we show that the densities $W_{\P_{\mathcal {H}}}$ and $\pi_{\P_{\D}}$ have the same functional form. Using the labeling in $\D$, the cliques of $\mathcal{H}$ are $C_{1}=\preceq
1\succeq$, ~$C_{2}=\preceq 2\succeq$. The only separator is $S_{2}=\prec 1\succ$.  Thus  $c_{1}=2,\; c_{2}=2, \; s_{2}=1$. Replacing these
in the corresponding Markov ratio that appears in the $\W_{\P_{\G}}$ distribution we obtain
\begin{equation*}
\frac{\det\left( \Sigma_{\preceq 1\succeq} \right)^{\alpha_{1}+\frac{3}{2}}\det\left(\Sigma_{\preceq 2\succeq}\right)^{\alpha_{2}+\frac{3}{2}}}
{\det\left( \Sigma_{\prec 1\succ}
\right)^{\beta_{2}+1}}=\dfrac{\left(D_{11}D_{33}\right)^{\alpha_{1}+\frac{3}{2}}\left(D_{22}D_{33}\right)^{\alpha_{2}+\frac{3}{2}}}{D_{33}^{\beta_2+1}}=D_{11}^{\alpha_{1}+\frac{3}{2}}D_{22}^{\alpha_{2}+\frac{3}{2}}D_{33}^{\alpha_{1}+\alpha_{2}
-\beta_{2}+2}.
\end{equation*}
Therefore, $W_{\P_{\mathcal {H}}}\left(\alpha,\beta, U^{H}\right)$ and $\pi_{\P_{\D}}\left(\gamma, U\right)$ have exactly the same functional form when $\gamma_{1}=\alpha_{1}+3/2$, $\gamma_{2}=\alpha_{2}+3/2$ and
$\gamma_{3}=\alpha_{1}+\alpha_{2} -\beta_{2}+2$. To identify the space of shape parameters for $\W_{\P_{\mathcal{H}}}$ we proceed to compute  
\[
\begin{array}{lll}
\rho_{[1]}=\alpha_{1},& \rho_{[2]}=\alpha_{2},&  \rho_{[3]}=\alpha_{1}+\alpha_{2}-\beta_{2},\\
n_{[1]}=1, &n_{[2]}=1,& n_{[3]}=1.
\end{array}
\]
Now by using Theorem 3.2 in \cite{L07} we obtain  $\mathcal{B}=\left\{\left(\alpha_1,\alpha_2,\beta_3\right): \alpha_1<0,\: \alpha_2<0,\:\beta_2-\alpha_2-\alpha_1<3/2\right\}$. Th space of shape parameters
$\left(\gamma_1, \gamma_2,\gamma_3\right)$  for $\pi_{\P_{\D}}$ is easily determined by inequalities $\gamma_{j}<pa_{j}/2+1$, which yields the same inequalities $\alpha_1<0,\: \alpha_2<0$ and $\beta_2-\alpha_2-\alpha_1<3/2$. This shows that up to re-parametrization $W_{\P_{\mathcal {H}}}\left(\alpha,\beta, U^{H}\right)$ and $\pi_{\P_{\D}}\left(\gamma, U\right)$ are the same distributions. In the next example we shall show that the family of DAG Wisharts $\pi_{\P_{\D}}\left(\gamma, U\right)$ strictly contains the  family of Type II Wisharts $W_{\P_{\mathcal {H}}}\left(\alpha,\beta, U^{H}\right)$.
\end{Ex}
\begin{figure}[ht]
\centering
\subfigure[]{
\includegraphics[width=2.8cm]{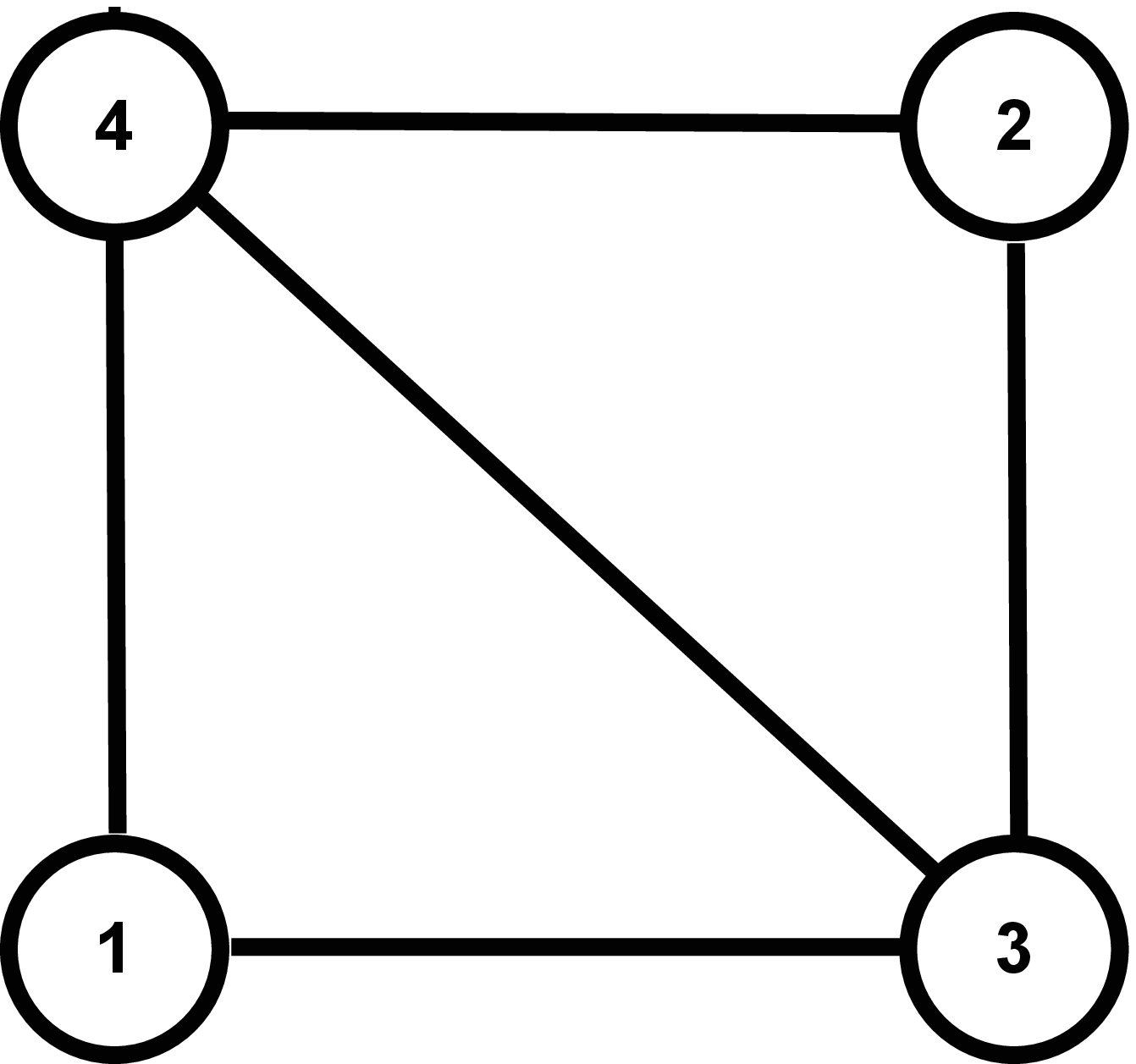}
\label{fig:subfiga}
}
\hspace{2cm}
\subfigure[]{
\includegraphics[width=3.6cm]{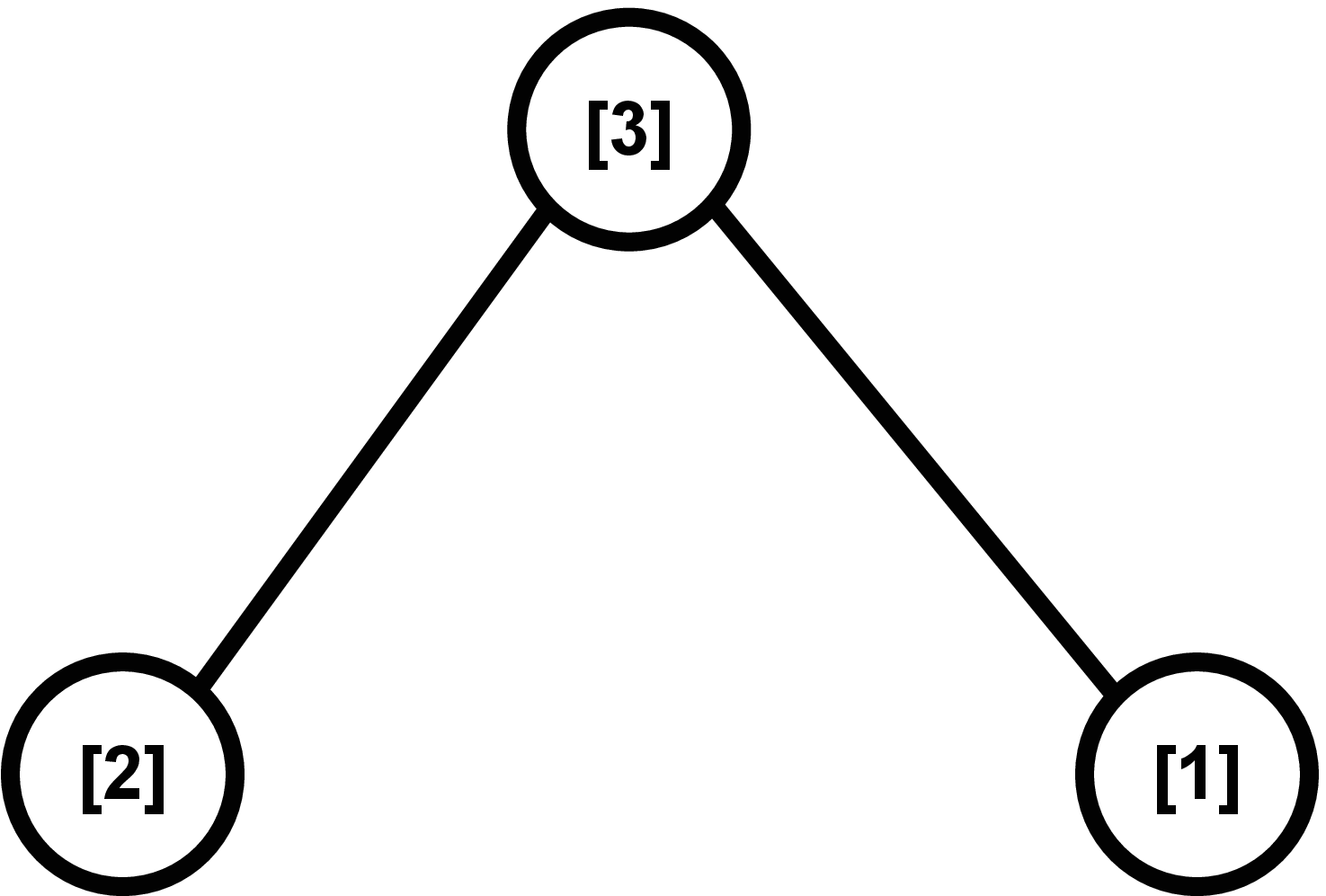}
\label{fig:subfigb}
}
\hspace{2cm}
\subfigure[]{
\includegraphics[width=2.8cm]{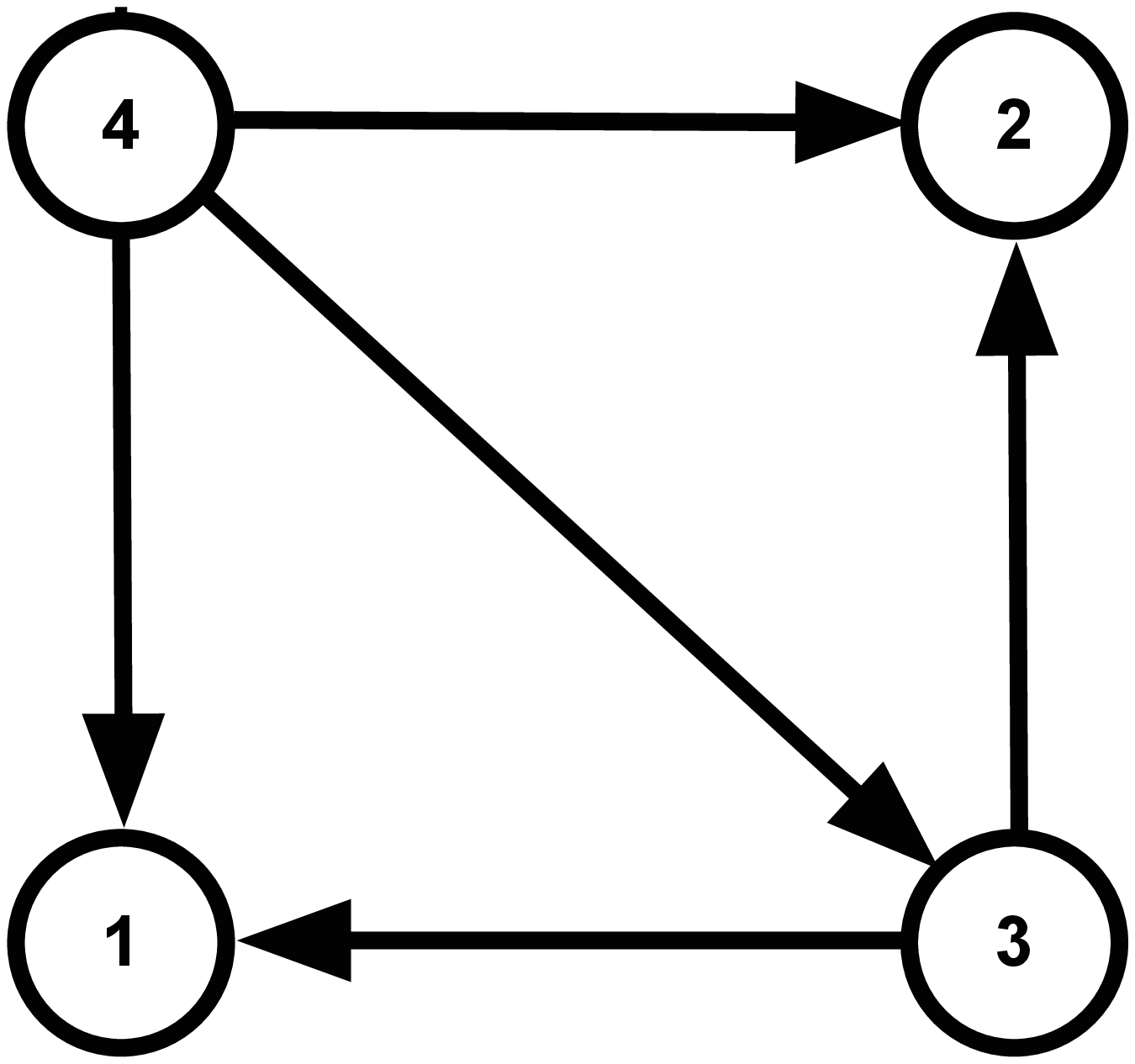}
\label{fig:subfigc}
}
\caption{A homogeneous graph, its Hasse diagram, and a DAG version. }
\label{fig:label}
\end{figure}
\begin{Ex} Let $\mathcal{H}$ be the homogeneous graph given in Figure \ref{fig:subfiga}. The
Hasse tree of $\mathcal{H}$ is given in Figure \ref{fig:subfigb}. Let $\D$ be the perfect transitive DAG version given in Figure \ref{fig:subfigc}. The cliques of  $\mathcal{H}$  are $C_{1}=\, \preceq 1\succeq$ and $C_{2}=\, \preceq 2\succeq$. The only separator is
$S_{2}=\,\prec 1\succ$. Thus  $c_{1}=3, \; c_{2}=3$ and $s_{2}=2$. Substituting these in the corresponding Markov ratio which appears in the $\W_{\P_{\G}}$ distribution we obtain
  \[
\dfrac{\det\left(\Sigma_{\preceq 1\succeq}\right)^{\alpha_{1}+2}\det\left(\Sigma_{\preceq 2\succeq}\right)^{\alpha_{2}+2}}{\det\left(\Sigma_{\prec
1\succ}\right)^{\beta_{2}+\frac{3}{2}}}=
D_{11}^{\alpha_1+2}D_{22}^{\alpha_2+2}D_{33}^{\alpha_1+\alpha_2-\beta_2+\frac{5}{2}}D_{44}^{\alpha_1+\alpha_2-\beta_2+\frac{5}{2}}.
  \]
Therefore, the choice of $\gamma_1=\alpha_1+2$, $\gamma_2=\alpha_2+2$ and $\gamma_3=\gamma_4=\alpha_1+\alpha_2-\beta_2+5/2$ the density $\pi_{\P_{\D}}\left(\gamma, U\right)$
is equal to $W_{\P_{\mathcal {G}}}\left(\alpha, \beta, U\right)$. In order to apply Theorem 3.2 in \cite{L07} for identifying the space of shape parameters $\mathcal{B}$, first we compute
 \[
 \begin{array}{lll}
 \rho_{[1]}=\alpha_{1},&  \rho_{[2]}=\alpha_{2},&  \rho_{[3]}=\alpha_{1}+\alpha_{2}-\beta_{2}\\ 
 n_{[1]}=1,& n_{[2]}=1, & n_{[3]}=2.
 \end{array}
 \]
From this and after some calculations we obtain  $\left(\alpha_1,\alpha_2,\beta_2\right)\in \mathcal{B}$ if and only if
\[
\alpha_1<0,\: \alpha_2<0 \: \text{and}\: \alpha_1+\alpha_2-\beta_2<-3/2.
\]
 Alternatively, with much less computation, from the inequalities $\gamma_{j}<pa_{j}/2+1$ we can identify $\mathcal{B}$.
 \end{Ex}
 \noindent

The examples above demonstrate that for a homogeneous graph $\mathcal{H}$, the Type II Wishart is a special case of the DAG Wishart for a perfect transitive DAG version of $\mathcal{H}$  . Furthermore, identifying the space of shape parameter $\mathcal{B}$ under the DAG Wishart family is computationally less expensive.
\begin{Rem}
Proposition \ref{prop:comparison_hom} for homogeneous graphs is stronger than Theorem \ref{thm:comparison} for decomposable graphs. The reason is that the latter guarantees the family of $\mathcal{W}_{\rm{P}_{\G}}$ distributions is a subfamily of DAG Wisharts $\pi_{\rm{P}_{\D}}$ only when $(\alpha, \beta)$ are restricted to $\mathcal{B}_{\mathcal{P}}\subset \mathcal{B}$. Proposition \ref{prop:comparison_hom} however guarantees that the family of $\mathcal{W}_{\rm{P}_{\G}}$ is a subfamily of DAG Wisharts $\pi_{\rm{P}_{\D}}$ on the whole parameter set $\mathcal{B}$. Also note that Proposition \ref{prop:comparison_hom} is not implied by Theorem \ref{thm:comparison}. To see this, consider the DAG $\D$ given by Figure \ref{fig:path-3c}. This DAG is a perfect, but non-transitive, DAG version of the homogeneous graph $A_3$ given by Figure \ref{fig:path-3a} and in fact induced by the perfect order $\mathcal{P}=(C_1:=\{2,3 \}, C_2:=\{1,2 \})$ of the cliques of $A_3$. It is easy to check that $B_{\mathcal{P}}\subsetneq \mathcal{B}$. Therefore using Theorem \ref{thm:comparison} for this homogeneous graph does not imply that the family of $\mathcal{W}_{\rm{P}_{\G}}$ is a subfamily of DAG Wisharts $\pi_{\rm{P}_{\D}}$ on the whole parameter set $\mathcal{B}$.
\end{Rem}
 \section{Counterexamples to part \textbf{II} of the LM conjecture}\label{sub:LM_partb}
We now return to the LM conjecture (II) in this section. Using the tools we have developed thus far, we proceed to obtain some counterexamples to show that Part (II) of the LM conjecture fails. In particular, we show that there exist decomposable graphs where the space of shape parameters $\mathcal{B}$ for the Type II Wisharts, over such a graph $\G$, strictly contains $\bigcup_{\mathcal{P}\in\mathrm{Ord}(\G) }B_{\mathcal{P}}$. Recall that a pair $\left(\alpha, \beta\right)\in\R^{r}\times\R^{r-1}$ belongs to $\mathcal{B}$  if and only if it satisfies both Equation \eqref{eq:B1} and Equation \eqref{eq:B2}. As we hinted in Remark \ref{rem:key}, if there exists a decomposable graph $\G$, necessarily non-homogeneous, that has a perfect order $\mathcal{P}$ such that the DAG version $\D$ induced by $\mathcal{P}$ yields $r_{\D}\ge 2$ (recall Definition \ref{def:ans}), then Equation \eqref{eq:B1} is satisfied on a set  that strictly contains $\bigcup_{\mathcal{P}\in\mathrm{Ord}(\G) }B_{\mathcal{P}}$. In the following examples we show that in such a situation, simultaneously for the same set, Equation \eqref{eq:B2} can be satisfied as well.
\begin{figure}[ht]
\centering
\subfigure[]{
\includegraphics[width=3.5cm]{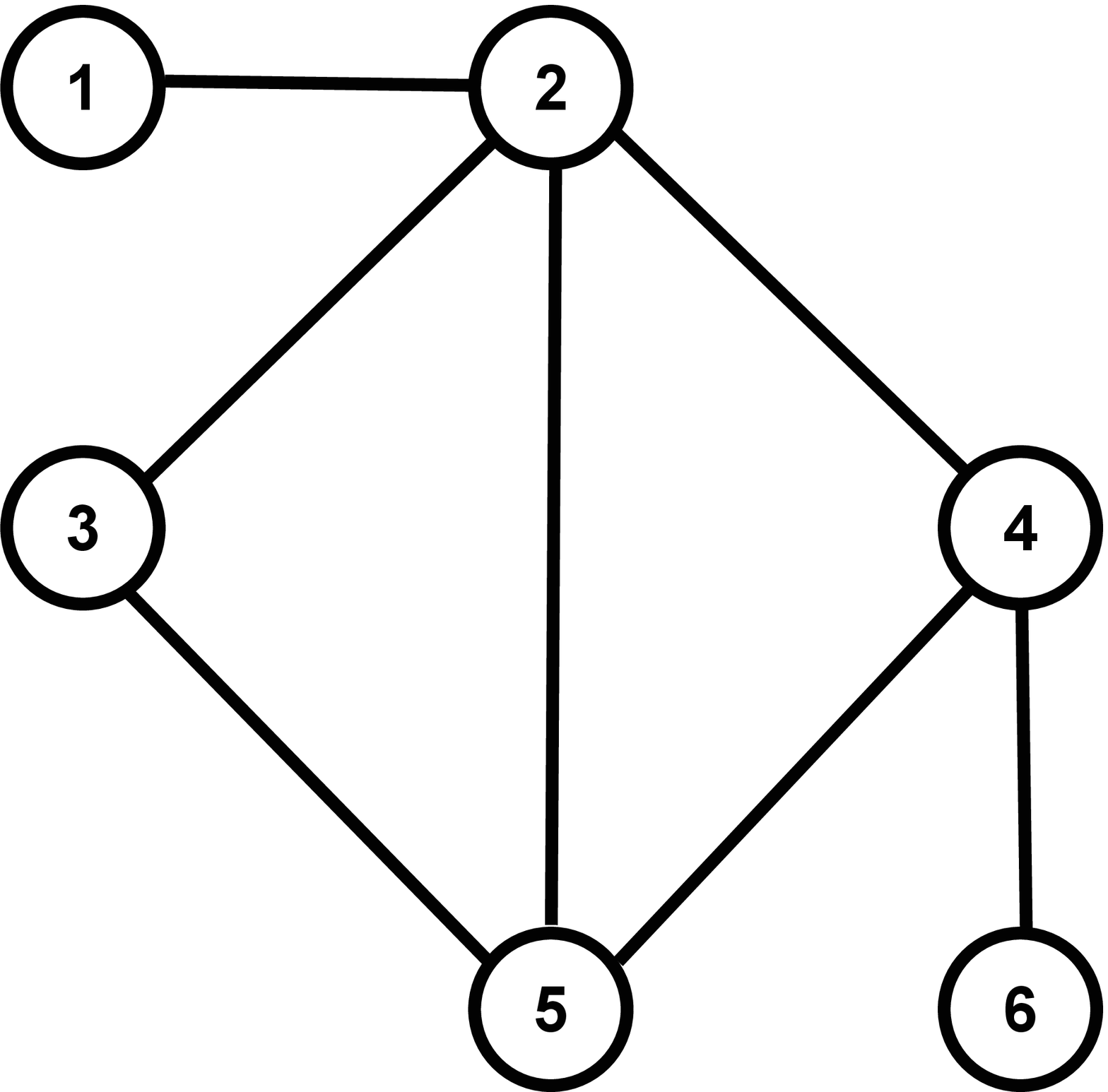}
\label{fig:c1}
}
\hspace{2cm}
\subfigure[]{
\includegraphics[width=3.51cm]{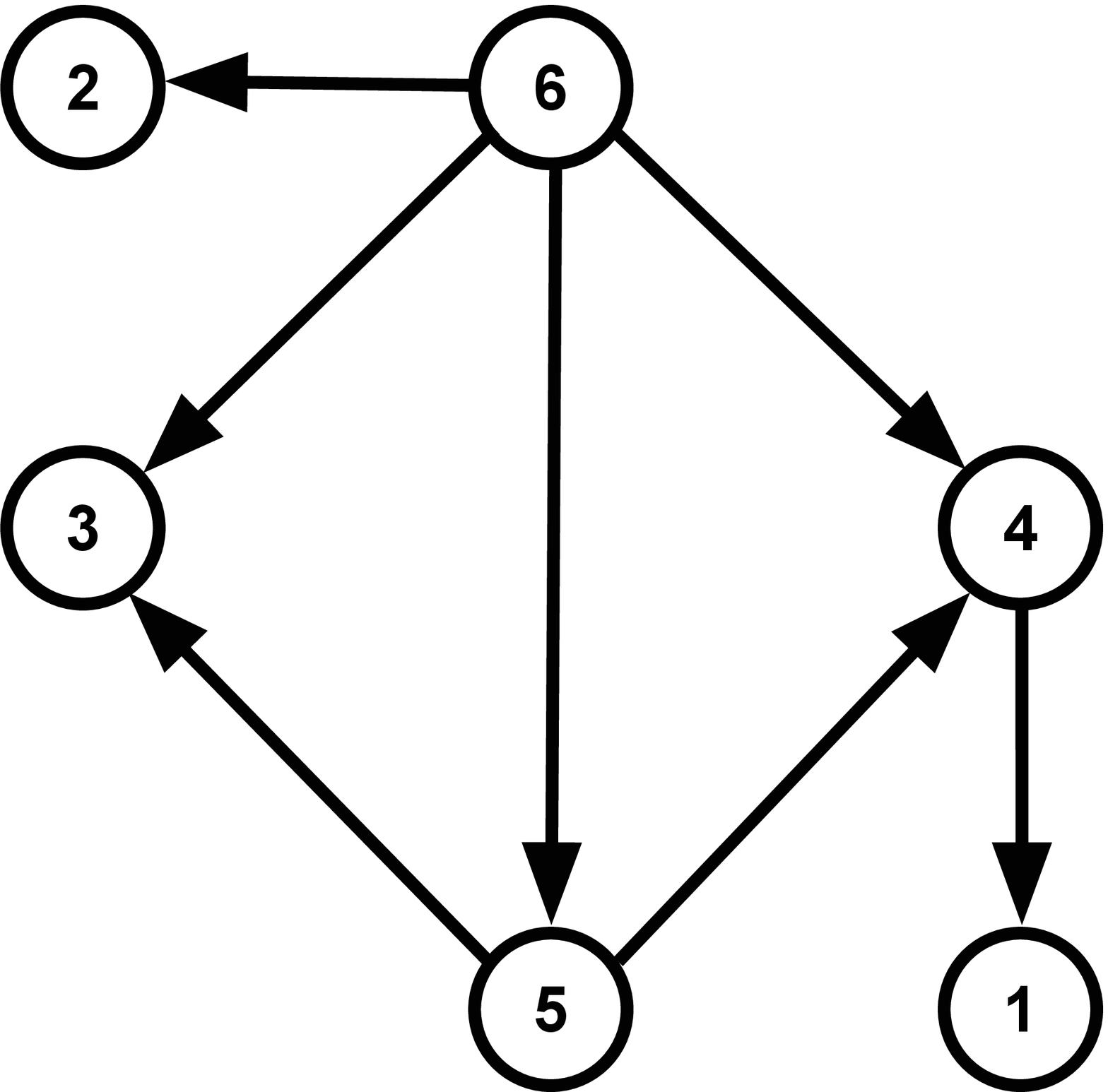}
\label{fig:cd1}
}
\caption{ First counterexample to the LM conjecture (II). }
\end{figure}
 \begin{Ex} \label{ex:LM2_a}
 
Let $\G$ be the decomposable graph given in Figure \ref{fig:c1}. Clearly, $\G$ is a non-homogeneous decomposable graph. Consider the perfect order
$\mathcal{P}=\left(C_{1}=\left\{2, 3, 5\right\}, C_{2}=\left\{2, 4, 5\right\}, C_{3}=\left\{4, 6 \right\}, C_{4}=\left\{1, 2 \right\}\right)$ of the cliques of $\G$. The separators of $\G$  are  $S_{2}=\left\{ 2, 5\right\}, S_{3}=\left\{ 4\right\}$ and $S_{4}=\left\{ 2\right\}$. It is easy to check that the DAG $\D$ given in Figure \ref{fig:cd1} is a directed version of $\G$  induced by $\mathcal{P}$. Using the labeling in $\D$ we have 
\[
\begin{array}{llll}
C_{1}=\preceq 3\succeq\;, &C_{2}=\preceq  4\succeq\; , &C_{3}=\preceq  2\succeq\; ,& C_{4}=\preceq  1\succeq\: \text{and}\\
& S_{2}=\prec  3\succ\; , &S_{3}=\prec  2\succ\;, &S_{4}=\prec  1\succ .  
 \end{array}
 \]
Note that for the DAG version $\D$ given by Figure \ref{ex:LM2_a} $r_{\D}=2$ since  $S_{3}$ and $S_{2}$ are both ancestral in $\D$. Thus by Proposition \ref{prop:B1} Equation \eqref{eq:B1} is satisfied on a set of $\left(\alpha, \beta \right)$ that is of dimension greater than or equal to $r+r_{\D}=6$. This is strictly greater than $5$, the dimension of $\bigcup_{\mathcal{P}\in\mathrm{Ord}(\G) }B_{\mathcal{P}}$. Hence it is clear that the LM conjecture (II) fails for this example. Nonetheless, for this specific example, we provide the reader with a self-contained proof. To begin with, we rewrite the Markov ratio term that appears in $\W_{\P_{\G}}$ as follows:
 \begin{align}\label{eq:markov_ratio_example1}
\notag&\frac{\prod_{j=1}^{4} \det\left( \Sigma_{C_{j}} \right)^{\alpha_{j}+\frac{c_{j}+1}{2}}}
{\prod_{j=2}^4 \det\left( \Sigma_{S_{j}} \right)^{\beta_{j}+\frac{s_{j}+1}{2}}}=\dfrac{\det\left(\Sigma_{\preceq 3\succeq}\right)^{\alpha_{1}+2}\det\left(\Sigma_{\preceq
4\succeq}\right)^{\alpha_{2}+2}\det\left(\Sigma_{\preceq 2\succeq}\right)^{\alpha_{3}+\frac{3}{2}}\det\left(\Sigma_{\preceq
1\succeq}\right)^{\alpha_{4}+\frac{3}{2}}}{\det\left(\Sigma_{\prec 3\succ}\right)^{\beta_{2}+\frac{3}{2}}\det\left(\Sigma_{\prec
2\succ}\right)^{\beta_{3}+1}\det\left(\Sigma_{\prec 1\succ}\right)^{\beta_{4}+1}}\\
\notag&=\dfrac{D_{33}^{\alpha_{1}+2}D_{55}^{\alpha_{1}+2}D_{66}^{\alpha_{1}+2}D_{44}^{\alpha_{2}+2}D_{55}^{\alpha_{2}+2}D_{66}^{\alpha_{2}+2}D_{22}^{\alpha_{3}+\frac{3}{2}}D_{66}^{\alpha_{3}+\frac{3}{2}}D_{11}^{\alpha_{4}+\frac{3}{2}}\Sigma_{44}^{\alpha_{4}+\frac{3}{2}}}{D_{55}^{\beta_{2}+\frac{3}{2}}D_{66}^{\beta_{2}+\frac{3}{2}}D_{66}^{\beta_{3}+1}\Sigma_{44}^{\beta_{4}+1}}\\
&=D_{11}^{\alpha_{4}+\frac{3}{2}}D_{22}^{\alpha_{3}+\frac{3}{2}}D_{33}^{\alpha_{1}+2}D_{44}^{\alpha_{2}+2}D_{55}^{\alpha_{1}+\alpha_{2}-\beta_{2}+\frac{5}{2}}D_{66}^{\alpha_{1}+\alpha_{2}+\alpha_{3}-\beta_{2}-\beta_{3}+3}\Sigma_{44}^{\alpha_{4}-\beta_{4}+\frac{1}{2}}.
\end{align}
Let $\gamma_{j}$ be the exponent of $D_{jj}$ in Equation \eqref{eq:markov_ratio_example1}. Then $\int_{\P_{\mathcal
{G}}}\omega_{\P_{\G}}\left(\alpha,\beta, U^{E}, d\Omega\right)<\infty$ for every $\left(\alpha, \beta\right)\in \R^{4}\times
\R^{3}$ such that
\begin{equation}\label{eq:example1_convergence}
\forall j\:\: \alpha_{j}<0 ,\: \alpha_{1}+\alpha_{2}-\beta_{2}+1
<0,\: \alpha_{1}+\alpha_{2}+\alpha_{3}-\beta_{2}-\beta_{3}+2<0,\:\alpha_{4}-\beta_{4}+\frac{1}{2}=0.
 \end{equation}
Now we show that for each $\left(\alpha,\beta\right)$ that satisfies Equation \eqref{eq:example1_convergence}, not only Equation \eqref{eq:B1} is satisfied, but also Equation \eqref{eq:B2} is satisfied, i.e., 
\[
\int_{\P_{\mathcal {G}}}\omega_{\P_{\G}}\left(\alpha,\beta, U^{E},
d\Omega\right)/H_{\G}\left(\alpha,\beta,U^{E}\right) \: \text{is functionally independent of $U^{E}$}.
\]
 By Equation \eqref{eq:markov_ratio_example1} and Equation \eqref{eq:normalizing_constant_rewritten} we have
 \begin{align*}
\notag&\int_{\P_{\mathcal {G}}}\omega_{\P_{\G}}\left(\alpha,\beta, U^{E}, d\Omega\right)\propto \prod_{j=1}^{6}\left(U_{jj|\prec j\succ}\right)^{\gamma_{j} -\frac{pa_{j}}{2} -1}\det\left( U_{\prec j\succ}\right)^{-\frac{1}{2}}\\
\notag &=\left( U_{11|\prec 1\succ}\right)^{\alpha_{4}}\left( U_{22|\prec 2\succ}\right)^{\alpha_{3}}\left( U_{33|\prec 3\succ}\right)^{\alpha_{1}}\left( U_{44|\prec
4\succ}\right)^{\alpha_{2}}\det\left( U_{55|\prec 5\succ}\right)^{\alpha_{1}+\alpha_{2}-\beta_{2}+1}\left( U_{66|\prec
6\succ}\right)^{\alpha_{1}+\alpha_{2}+\alpha_{3}-\beta_{2}-\beta_{3}+2}\\
\notag&\times\det\left(U_{\prec 1\succ}\right)^{-\frac{1}{2}}\det\left(U_{\prec 2\succ}\right)^{-\frac{1}{2}}\det\left(U_{\prec 3 \succ}\right)^{-\frac{1}{2}}\det\left(U_{\prec
4\succ}\right)^{-\frac{1}{2}}\det\left(U_{\prec 5\succ}\right)^{-\frac{1}{2}}\\
\notag&=\left( U_{11|\prec 1\succ}\right)^{\alpha_{4}}\left( U_{22|\prec 2\succ}\right)^{\alpha_{3}}\left( U_{33|\prec 3\succ}\right)^{\alpha_{1}}\left( U_{44|\prec
4\succ}\right)^{\alpha_{2}}\left( U_{55|\prec 5\succ}\right)^{\alpha_{1}+\alpha_{2}-\beta_{2}+1}\left( U_{66|\prec
6\succ}\right)^{\alpha_{1}+\alpha_{2}+\alpha_{3}-\beta_{2}-\beta_{4}+2}\\
\notag&\times U_{44}^{-\frac{1}{2}}\left(U_{66|\prec 6\succ}\right)^{-\frac{1}{2}}\left(U_{55|\prec 5 \succ}\right)^{-\frac{1}{2}}\left(U_{66|\prec
6\succ}\right)^{-\frac{1}{2}}\left(U_{55|\prec 5\succ}\right)^{-\frac{1}{2}}\left(U_{66|\prec 6\succ}\right)^{-\frac{1}{2}}\left(U_{66|\prec 6\succ}\right)^{-\frac{1}{2}}\\
&=\left( U_{11|\prec 1\succ}\right)^{\alpha_{4}}\left( U_{22|\prec 2\succ}\right)^{\alpha_{3}}\left( U_{33|\prec 3\succ}\right)^{\alpha_{1}}\left( U_{44|\prec
4\succ}\right)^{\alpha_{2}}\left( U_{55|\prec 5\succ}\right)^{\alpha_{1}+\alpha_{2}-\beta_{2}}\left( U_{66|\prec
6\succ}\right)^{\alpha_{1}+\alpha_{2}+\alpha_{3}-\beta_{2}-\beta_{3}}U_{44}^{-\frac{1}{2}}.
 \end{align*}
 Now similar to our computation in Equation \eqref{eq:markov_ratio_example1} we can show 
 \begin{align*}
H_{\G}\left(\alpha,\beta, U^{E}\right)&=\left(U_{11|\prec 1\succ}\right)^{\alpha_{4}}\left(U_{22|\prec 2\succ}\right)^{\alpha_{3}}\left(U_{33|\prec
3\succ}\right)^{\alpha_{1}}\left(U_{44|\prec 4\succ}\right)^{\alpha_{2}}\\
&\times \left(U_{55|\prec 5\succ}\right)^{\alpha_{1}+\alpha_{2}-\beta_{2}}\left(U_{66|\prec
6\succ}\right)^{\alpha_{1}+\alpha_{2}+\alpha_{3}-\beta_{2}-\beta_{3}}U_{44}^{\alpha_{4}-\beta_{4}}.
 \end{align*}
  Therefore, we have
  \[
\int_{\P_{\mathcal {G}}}\omega_{\P_{\G}}\left(\alpha,\beta, U^{E},
d\Omega\right)/H_{\G}\left(\alpha,\beta,U^{E}\right)\propto
  U_{44}^{-\left(\alpha_{4}-\beta_{4}+\frac{1}{2}\right)}=1, \: \text{by Equation \eqref{eq:example1_convergence}}.
  \]
  This completes our first counterexample.
\end{Ex}
\begin{Ex}\label{ex:LM2_b}
Let $\G$ be the non-homogeneous graph given in \ref{fig:c2} and let $\D$ be the DAG given in \ref{fig:cd2}. One can readily check that $\D$ is a DAG version of $\G$ induced by 
\[
\mathcal{P}=\left(\preceq 4\succeq,\; \preceq 3\succeq,\; \preceq 6\succeq, \;\preceq 2\succeq, \;\preceq 1\succeq \right).
\]
Since there are $5$ cliques and $4$ separators $r=5$ and $\mathcal{B}\subset \R^{9}$. Therefore, the LM conjecture (II) here implies that the dimension of $\mathcal{B}$ is $6$. Now $r_{\D}=2$ as the separators $S_{2}=\;\prec 3\succ $ and
$S_{3}=\;\prec 7\succ$ are ancestral ( two other separators $S_4=\;\prec 5\succ$ and $S_5=\;\prec 1\succ $ are not ancestral). Therefore, by Proposition \ref{prop:B1}, Equation \eqref{eq:B1} is satisfied on a set  of dimension greater than or equal to $r+r_{\D}=7$. We leave it to the reader to show that with a similar calculation as in Example \ref{ex:LM2_a} Equation \eqref{eq:B2} is satisfied on this set. Thus the LM conjecture (II) fails.

\end{Ex}
\begin{figure}[ht]
\centering
\subfigure[]{
\includegraphics[width=4.5cm]{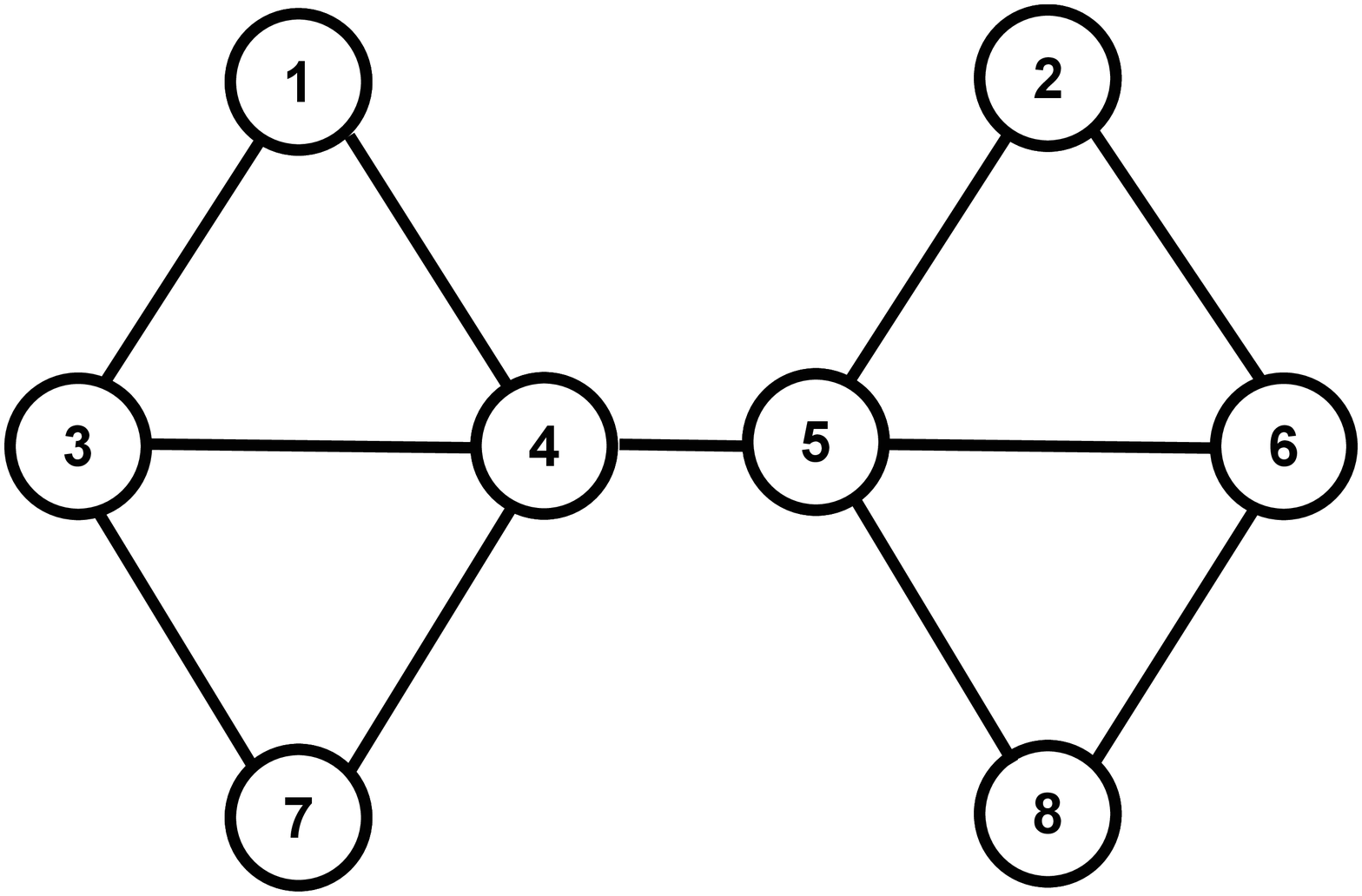}
\label{fig:c2}
}
\hspace{2cm}
\subfigure[]{
\includegraphics[width=4.5cm]{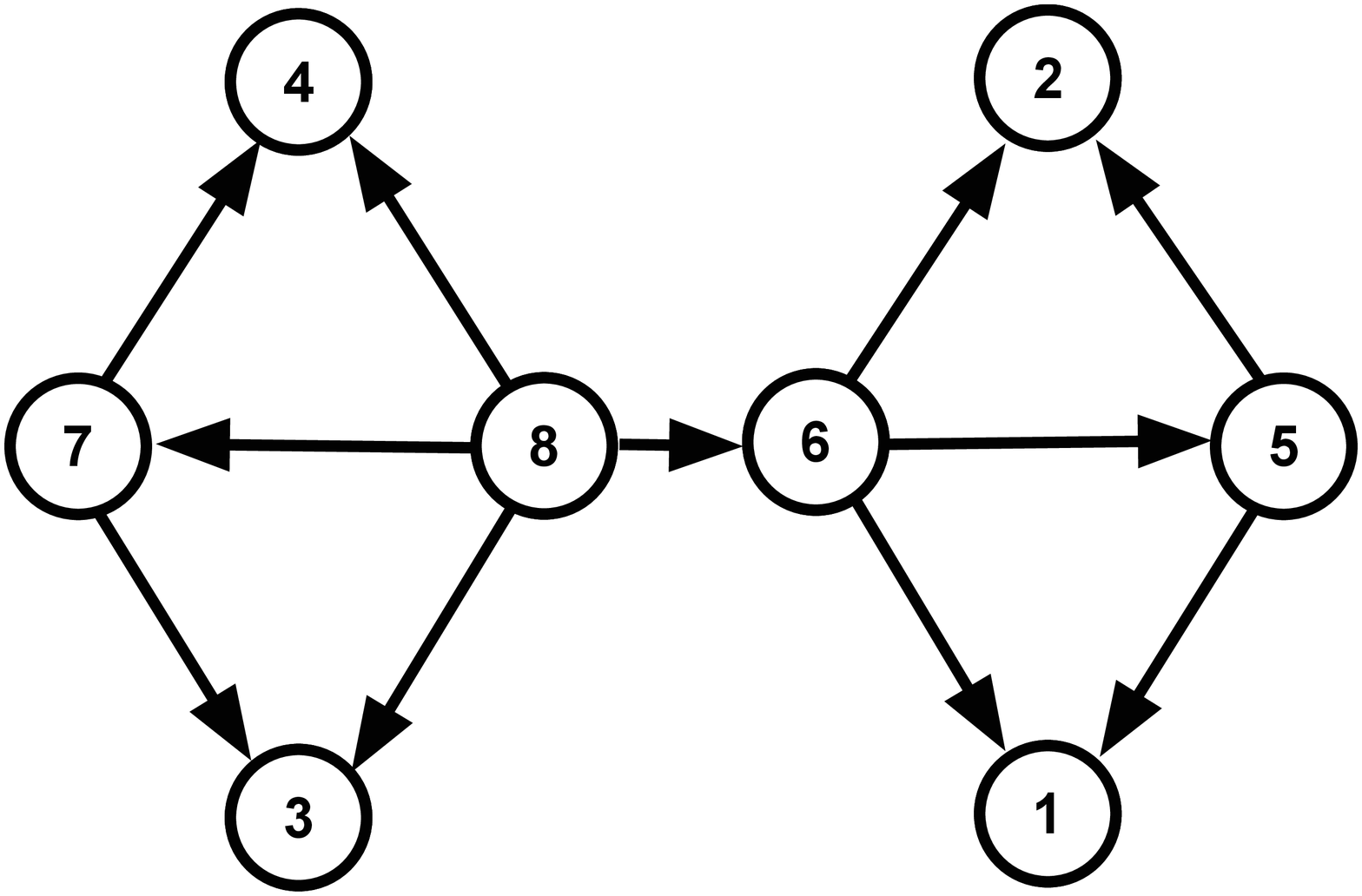}
\label{fig:cd2}
}
\caption{Second counterexample to the LM conjecture. }
\end{figure}
\begin{Rem}
Note that in Example \ref{ex:LM2_b} the separator $S_5=\:\preceq 5\succeq$. Therefore $\det\left(\Sigma_{S_5}\right)=\Sigma_{66}D_{55}$. By some calculations similar to those of Example \ref{ex:LM2_a}, we can show that the Markov ratio present in the corresponding Type II Wishart is
\begin{equation}\label{eq:rem}
\frac{\prod_{j=1}^{s} \det\left( \Sigma_{C_{j}} \right)^{\alpha_{j}+\frac{c_{j}+1}{2}}}
{\prod_{j=2}^4 \det\left( \Sigma_{S_{j}} \right)^{\beta_{j}+\frac{s_{j}+1}{2}}}\propto \prod_{j=1}^{8} D_{jj}^{\gamma_j}\Sigma_{66}^{\eta},
\end{equation}
where the exponents $\gamma_i$ and $\eta$ are some linear functions of $\alpha$ and $\beta$.  Now using Equation \eqref{eq:rem} one can show that both Equation \eqref{eq:B1} and Equation \eqref{eq:B2} are satisfied for every $(\alpha,\beta)\in \R^{9}$ subject to the constraint $\eta=\eta(\alpha,\beta)=0$. This shows that the dimension of $\mathcal{B}\ge 8$. On the other hand, we can see that for any DAG version of $\G$ always two of the separators will be be non-ancestral. Therefore, although Proposition \ref{prop:B1} identifies a subset of $\mathcal{B}$ that is strictly larger than the subset
identified by the LM conjecture (II) it does not identify the whole set $\mathcal{B}$, since for any DAG version of $\G$ Proposition \ref{prop:B1} will identify a subset of dimension $\le 7$.
\end{Rem}
\section{Counterexamples to part \textbf{I} of the LM conjecture}\label{sub:LM_parta}
In this section we once more use the theory we had developed in previous sections to produce counterexample to Part (I) of the LM conjecture. First we establish the following lemma.
\begin{lemma}\label{lem:mr}
Suppose that $\D$ is a perfect DAG version of $\G$. Then 
\begin{equation}\label{eq:mr}
\dfrac{\prod_{C\in\mathscr{C}_{\G}}\det\left(\Sigma_{C}\right)^{-\frac{|C|+1}{2}}}{\prod_{S\in\mathscr{S}_{\G}}\det\left(\Sigma_{S}\right)^{-\frac{|S|+1}{2}}}=\prod_{j\in V}\left(\Sigma_{jj|\prec
j\succ}\right)^{-\frac{pa_{j}+2}{2}}\det\left(\Sigma_{\prec j\succ}\right)^{-\frac{1}{2}}.
\end{equation}
\end{lemma}
\begin{proof}

We shall proceed by the method of mathematical induction. Suppose this is true for any decomposable graph with number of vertices less than $p$ and we prove the lemma for $|V|=p$. The
equality trivially holds when $p=1$. Let us therefore assume that $p>1$. As before, let $r$ be the number of the cliques and consider the following cases.
\begin{itemize}
\item[1)] Suppose that $r=1$, i.e., $\G$ is complete. We write
\begin{align*}
\dfrac{\prod_{C\in\mathscr{C}_{\G}}\det\left(\Sigma_{C}\right)^{-\frac{|C|+1}{2}}}{\prod_{S\in\mathscr{S}_{\G}}\det\left(\Sigma_{S}\right)^{-\frac{|S|+1}{2}}}&=\det\left(\Sigma\right)^{-\frac{p+1}{2}}\\
&=\left(\Sigma_{11|\prec 1\succ}\det\left(\Sigma_{\prec 1\succ}\right)\right)^{-\frac{p+1}{2}}\\
&=\left(\Sigma_{11|\prec 1\succ}\right)^{-\frac{pa_{1}+2}{2}}\det\left(\Sigma_{\prec 1\succ}\right)^{-\frac{1}{2}}\det\left(\Sigma_{\prec 1\succ}\right)^{-\frac{p}{2}}\\
&=\left(\Sigma_{11|\prec 1\succ}\right)^{-\frac{pa_{1}+2}{2}}\det\left(\Sigma_{\prec 1\succ}\right)^{-\frac{1}{2}}\prod_{2}^{p}\left(\Sigma_{jj|\prec
j\succ}\right)^{-\frac{pa_{j}+2}{2}}\det\left(\Sigma_{\prec j\succ}\right)^{-\frac{1}{2}},
\end{align*}
where the last equality uses the induction hypothesis for the induced graph $G_{\prec 1\succ}$.
\medskip
 \item[2)] Suppose that $r\geq 2$. Note that by Lemma \ref{lem:combine} we can assume that $\D$ is induced by some perfect order $\mathcal{P}=\left(C_{1},\ldots, C_{r}\right)$. In particular, there exists a vertex in $R_{r}$ that has no child. Thus, without loss of generality, assume that the vertices in  $\D$ are labeled such that $1\in R_{r}$. We now consider two cases:\\
~a)~ If the residual $R_{r}=\left\{1\right\}$, then $\left(C_{1},\ldots,C_{r-1}\right)$ and $\D_{V\setminus\left\{1\right\}}$ are, respectively, a perfect order of the cliques and a perfect DAG version of $\G_{V\setminus\left\{1\right\}}$. Moreover,  $S_{r}=\prec 1\succ$. Now it follows that  
\begin{align*}
&\dfrac{\prod_{j=1}^{r}\det\left(\Sigma_{C_{j}}\right)^{-\frac{c_{j}+1}{2}}}{\prod_{j=2}^{r}\det\left(\Sigma_{S_{j}}\right)^{-\frac{s_{j}+1}{2}}}=\dfrac{\prod_{j=1}^{r-1}\det\left(\Sigma_{C_{j}}\right)^{-\frac{c_{j}+1}{2}}}{\prod_{j=2}^{r-1}\det\left(\Sigma_{S_{j}}\right)^{-\frac{s_{j}+1}{2}}}\det\left(\Sigma_{R_{r}|S_{r}}\right)^{\frac{c_{r}+1}{2}}\det\left(\Sigma_{S_{r}}\right)^{\frac{s_{r}-c_{r}}{2}}\\
&=\prod_{j=2}^{p}\left(\Sigma_{jj|\prec j\succ}\right)^{-\frac{pa_{j}+2}{2}}\det\left(\Sigma_{\prec j\succ}\right)^{-\frac{1}{2}}\left(\Sigma_{11|\prec
1\succ}\right)^{-\frac{pa_{1}+2}{2}}\det\left(\Sigma_{\prec 1\succ}\right)^{-\frac{1}{2}}\\
&=\prod_{j=1}^{p}\left(\Sigma_{jj|\prec j\succ}\right)^{-\frac{pa_{j}+2}{2}}\det\left(\Sigma_{\prec j\succ}\right)^{-\frac{1}{2}}.
\end{align*}

~b)~ If the residual $R_{r}$ has more than one element, then $\left(C_{1},\ldots,C_{r-1},C_{r}\setminus\left\{1\right\}\right)$ is a perfect order of the cliques of $\G_{V\setminus\left\{1\right\}}$ with associated separators $S_{2},\ldots,S_{r}$. Using the induction hypothesis we obtain
\begin{align*}
&\prod_{j=1}^{p}\left(\Sigma_{jj|\prec j\succ}\right)^{-\frac{pa_{j}+2}{2}}\det\left(\Sigma_{\prec j\succ}\right)^{-\frac{1}{2}}=\left(\Sigma_{11|\prec
1\succ}\right)^{-\frac{pa_{1}+2}{2}}\det\left(\Sigma_{\prec j\succ}\right)^{-\frac{1}{2}}\prod_{j=2}^{p}\left(\Sigma_{jj|\prec
j\succ}\right)^{-\frac{pa_{j}+2}{2}}\det\left(\Sigma_{\prec j\succ}\right)^{-\frac{1}{2}}\\
&=\Sigma_{11|C_{r}\setminus\left\{ 1\right\}}^{-\frac{c_{r}+1}{2}}\det\left(\Sigma_{C_{r}\setminus \left\{1\right\}}\right)^{-\frac{1}{2}}\det\dfrac{\prod_{j=1}^{r-1} \det\left(\Sigma_{C_{j}}\right)^{-\frac{c_{j}+1}{2} }
}{\prod_{j=2}^{r}\det\left(\Sigma_{S_{j}}\right)^{-\frac{s_{j}+1}{2}}}\det\left( \Sigma_{C_{r}\setminus\left\{ 1\right\}}\right)^{-\frac{c_{r}}{2}}\\
&=\dfrac{\prod_{j=1}^{r-1}\det\left(\Sigma_{C_{j}}\right)^{-\frac{c_{j}+1}{2}}}{\prod_{j=2}^{r}\det\left(\Sigma_{S_{j}}\right)^{-\frac{s_{j}+1}{2}}}\det\left(\Sigma_{C_{r}}\right)^{-\frac{c_{r}+1}{2}}\\
&=\dfrac{\prod_{j=1}^{r}\det\left(\Sigma_{C_{j}}\right)^{-\frac{c_{j}+1}{2}}}{\prod_{j=2}^{r}\det\left(\Sigma_{S_{j}}\right)^{-\frac{s_{j}+1}{2}}}
\end{align*}
\end{itemize}
\end{proof}
\begin{Rem}
The Markov ratio in the right-hand-side of the Equation \eqref{eq:mr} is the squared root of the Jacobian of the inverse mapping
$\left(\Sigma^{E}\mapsto\Sigma^{-1}\right):\Q_{\G}\rightarrow \P_{\G}$. It can be shown directly that the left-hand-side of the Equation \eqref{eq:mr} is also is the squared root of the Jacobian of the inverse mapping $Q_{\D}\mapsto P_{\D}$ (see \cite{BR11}).\\
 We illustrate the result of Lemma \ref{lem:mr} for the decomposable graph $\G$  and its perfect DAG version $\D$  given in Figure \ref{fig:c1} and  Figure \ref{fig:cd1}, respectively.
\end{Rem}
\begin{Ex}
By using the same perfect order of $\mathcal{P}$  as in Example \ref{ex:LM2_a} the corresponding Markov
ratio becomes
 \begin{align*}
&\frac{\prod_{j=1}^{4} \det\left( \Sigma_{C_{j}} \right)^{-\frac{c_{i}+1}{2}}}
{\prod_{j=2}^4 \det\left( \Sigma_{S_{j}} \right)^{-\frac{s_{j}+1}{2}}}= \dfrac{\det\left(\Sigma_{\preceq 3\succeq}\right)^{-2}\det\left(\Sigma_{\preceq
4\succeq}\right)^{-2}\det\left(\Sigma_{\preceq 2\succeq}\right)^{-3/2}\det\left(\Sigma_{\preceq 1\succeq}\right)^{-3/2}}{\det\left(\Sigma_{\prec 3\succ}\right)^{-3/2}\det\left(\Sigma_{\prec
2\succ}\right)^{-1}\det\left(\Sigma_{\prec 1\succ}\right)^{-1}}\\
&=\dfrac{\left(\Sigma_{33|\prec 3\succ}\right)^{-2}\det\left(\Sigma_{\prec 3\succ}\right)^{-2}\left(\Sigma_{44|\prec 4\succ}\right)^{-2}\det\left(\Sigma_{\prec 4\succ}\right)^{-2}\left(\Sigma_{22|\prec
2\succ}\right)^{-3/2}\det\left(\Sigma_{\prec 2\succ}\right)^{-3/2}\left(\Sigma_{11|\prec 1\succ}\right)^{-3/2}\det\left(\Sigma_{\prec 1\succ}\right)^{-3/2}}{\det\left(\Sigma_{\prec
3\succ}\right)^{-3/2}\det\left(\Sigma_{\prec 2\succ}\right)^{-1}\det\left(\Sigma_{\prec 1\succ}\right)^{-1}}\\
&=\left(\Sigma_{33|\prec 3\succ}\right)^{-2}\det\left(\Sigma_{\prec 3\succ}\right)^{-1/2}\left(\Sigma_{44|\prec 4\succ}\right)^{-2}\det\left(\Sigma_{\prec 4\succ}\right)^{-1/2}\left(\Sigma_{55|\prec
5\succ}\right)^{-3/2}\det\left(\Sigma_{\prec 5\succ}\right)^{-1/2}\left(\Sigma_{22|\prec 2\succ}\right)^{-3/2}\\
&\times \det\left(\Sigma_{\prec 2\succ}\right)^{-1/2}\left(\Sigma_{11|\prec
1\succ}\right)^{-3/2}\det\left(\Sigma_{\prec 1\succ}\right)^{-1/2}\Sigma_{66}^{-1}\\
&=\prod_{j=1}^{6}\left(\Sigma_{jj|\prec j\succ}\right)^{-\frac{pa_{j}+2}{2}}\det\left(\Sigma_{\prec j\succ}\right)^{-1/2}.
\end{align*}
\end{Ex}
Next we use Lemma \ref{lem:mr} to prove an analog of Proposition \ref{prop:B1} for the Letac-Massam Type I Wisharts.
\begin{proposition}\label{prop:A1}
Let $\G$ be a non-complete (decomposable) graph. Assume that $\D$ is a DAG version of $\G$ induced by $\mathcal{P}$ such that $S_{2}$ is ancestral. Let $\mathscr{S}_{\G}^{\D}$ and $r_{\D}$ be as those defined in Definition \ref{def:ans}. Then the dimension of the set of $\left(\alpha,\beta\right)\in \R^{r}\times \R^{r-1}$ that satisfies  Equation \eqref{eq:A1} is greater than or equal to $r+r_{\D}$.
\end{proposition}
\begin{proof}
First note that  by Lemma \ref{lem:combine} assuming the existence of a such $\D$  is not vacuous. Now using Lemma \ref{lem:mr}, we begin with rewriting the density of the Type I Wishart to obtain an expression in terms of relationships in the directed version $\D$. We proceed as follows.
\begin{align}\label{eq:exp}
\notag
W_{Q_G}\left(\alpha,\beta,U^{E}\right)&\propto \exp\left\{-\tr\left(\Sigma
U^{-1}\right)\right\}H_{\G}\left(\alpha-\frac{c+1}{2},\beta-\frac{s+1}{2}, \Sigma^{E}\right)\\
\notag
&=\exp\left\{-\tr\left(\Sigma U^{-1}\right)\right\}\dfrac{\prod_{j=1}^{r} \det\left( \Sigma_{C_{j}} \right)^{\alpha_{j}}}
{\prod_{j=2}^{r} \det\left( \Sigma_{S_{j}} \right)^{\beta_{j}}}\dfrac{\prod_{j=1}^{r} \det\left( \Sigma_{C_{j}} \right)^{-\frac{c_{j}+1}{2}}}
{\prod_{j=2}^{r} \det\left( \Sigma_{S_{j}} \right)^{-\frac{s_{j}+1}{2}}}\\
\notag
&=\exp\left\{-\tr\left(\Sigma U^{-1}\right)\right\}\det\left(\Sigma_{\left(C_{1}\setminus
S_{2}\right)|S_{2}}\right)^{\alpha_{1}}\det\left(\Sigma_{S_{2}}\right)^{\alpha_{1}}\prod_{j=2}^{r}\det\left(\Sigma_{R_{j}|S_{j}}\right)^{\alpha_{j}}\\
\notag
&\times\prod_{j=2}^{r}\det\left(\Sigma_{S_{j}}\right)^{\alpha_{j}-\beta_{j}}\prod_{j=1}^{p}\left(\Sigma_{jj|\prec j\succ}\right)^{-\frac{pa_{j}+2}{2}}\det\left(\Sigma_{\prec
j\succ}\right)^{-\frac{1}{2}}\\
\notag
&=\exp\left\{-\tr\left(\Sigma U^{-1}\right)\right\}\det\left(\Sigma_{\left(C_{1}\setminus S_{2}\right)|S_{2}}\right)^{\alpha_{1}}\prod_{j=3}\det\left(\Sigma_{R_{j}|S_{j}}\right)^{\alpha_{j}}\\
&\times\prod_{S\in \mathscr{S}_{\G}\setminus{S_{2}}}^{r}\det\left(\Sigma_{S}\right)^{\sum \left(\alpha_{j}: j\in J\left(\mathcal{P},
S\right)\right)-\nu\left(S\right)\beta\left(S\right)}\prod_{j=1}^{p}\left(\Sigma_{jj|\prec j\succ}\right)^{-\frac{pa_{j}+2}{2}}\det\left(\Sigma_{\prec j\succ}\right)^{-\frac{1}{2}}.
\end{align}
Recall that in the proof of Theorem \ref{thm:comparison} it was showed that  $\det\left(\Sigma_{C_{1}\setminus S_{2}|S_{2}}\right)$, $\det\left(\Sigma_{R_{j}|S_{j}}\right)$\; for $2\leq j\leq r$,  and $\det\left(\Sigma_{S_j}\right)$ for each  $S\in \mathscr{S}_{\G}^{\D}$ can be written as products of $\Sigma_{\ell\ell|\prec \ell\succ}$ raised to some powers. Now define
\[
C_{\mathcal{P}}:= \left\{ \left(\alpha,\beta\right)\in \R^{r}\times \R^{r-1}: \sum\left(\alpha_{j}:\; j\in J\left(\mathcal{P}, S\right)\right)-\nu\left(S\right)\beta\left(S\right)=0 \quad \forall
S\notin \mathscr{S}_{\G}^{\D}\right\}.
\]
 Clearly, for each $\left(\alpha,\beta\right)\in C_{\mathcal{P}}$ the expression in Equation \eqref{eq:exp} reduces to
\begin{equation}\label{eq:Riesz} 
\exp\left\{-\tr\left(\Sigma U^{-1}\right)\right\}\prod_{j=1}^{p} \left(\Sigma_{jj|\prec j\succ}\right)^{\lambda_{j}}\prod_{j=1}^{p}\left(\Sigma_{jj|\prec
j\succ}\right)^{-\frac{pa_{j}+2}{2}}\det\left(\Sigma_{\prec j\succ}\right)^{-\frac{1}{2}},
\end{equation}
where $\lambda_{j}=\lambda_{j}\left(\alpha, \beta\right)$ is an affine combination of the components of $\alpha$ and $\beta$. The expression in Equation \eqref{eq:Riesz} is the non-normalized version of the density of the generalized Riesz distribution on $Q_G$, defined in \cite{A10}, and is integrable, by Equation (18) in \cite{A10},  if and only if $\lambda_{j}> pa_{j}/2$\; for each $j=1,\ldots,\, p$.
\end{proof}
An important observation in Proposition \ref{prop:A1} is that if $\G$ is chosen such that $r_{\D}>1$, then Part (I) of the LM conjecture may fail.  In fact we can show that the same decomposable graphs given in Figure \ref{fig:c1} and Figure \ref{fig:c2} provide two counterexamples to Part (I) of the LM conjecture. Since the calculations are very similar, we provide details only for the second graph.\\
\begin{Ex}
Consider the graph $\G$  given in Figure  \ref{fig:c2} and its DAG version  given in Figure \ref{fig:cd2}. Note that the LM conjecture (I) implies that the dimension of the parameter set $\mathcal{A}$ is $6$ since $r=5$. Now $\D$ is induced by
$\mathcal{P}=\left(\preceq 4\succeq, \preceq 3\succeq, \preceq 6\succeq, \preceq 2\succeq, \preceq 1\succeq\right)$. We proceed as follows.
\begin{align}\label{eq:markov_ratio_example_b}
\notag&\dfrac{\prod_{j=1}^{5} \det\left( \Sigma_{C_{j}} \right)^{\alpha_{j}-\frac{c_{j}+1}{2}}}
{\prod_{j=2}^{5} \det\left( \Sigma_{S_{j}} \right)^{\beta_{j}-\frac{s_{j}+1}{2}}}=
\dfrac{ \det\left( \Sigma_{\preceq 4\succeq} \right)^{\alpha_{1}}\det\left( \Sigma_{\preceq 3\succeq} \right)^{\alpha_{2}}\det\left( \Sigma_{\preceq 6\succeq}
\right)^{\alpha_{3}}\det\left( \Sigma_{\preceq 2\succeq} \right)^{\alpha_{4}}\det\left( \Sigma_{\preceq 1\succeq} \right)^{\alpha_{5}}}
{ \det\left( \Sigma_{\prec 3\succ} \right)^{\beta_{2}}\det\left( \Sigma_{\prec 6\succ} \right)^{\beta_{3}}\det\left( \Sigma_{\prec 5\succ} \right)^{\beta_{4}}\det\left(
\Sigma_{\prec 1\succ} \right)^{\beta_{5}}}\\
\notag&\times
\dfrac{ \det\left( \Sigma_{\preceq 4\succeq} \right)^{-2}\det\left( \Sigma_{\preceq 3\succeq} \right)^{-2}\det\left( \Sigma_{\preceq 6\succeq} \right)^{-\frac{3}{2}}\det\left(
\Sigma_{\preceq 2\succeq} \right)^{-2}\det\left( \Sigma_{\preceq 1\succeq} \right)^{-2}}
{ \det\left( \Sigma_{\prec 3\succ} \right)^{-\frac{3}{2}}\det\left( \Sigma_{\prec 6\succ} \right)^{-1}\det\left( \Sigma_{\prec 5\succ} \right)^{-1}\det\left( \Sigma_{\prec 1\succ}
\right)^{-\frac{3}{2}}}\\
\notag&=\dfrac{D_{44}^{\alpha_{1}}D_{77}^{\alpha_{1}}D_{88}^{\alpha_{1}}D_{33}^{\alpha_{2}}D_{77}^{\alpha_{2}}D_{88}^{\alpha_{2}}D_{66}^{\alpha_{3}}D_{88}^{\alpha_{3}}D_{22}^{\alpha_{4}}D_{55}^{\alpha_{4}}\Sigma_{66}^{\alpha_{4}}D_{11}^{\alpha_{5}}D_{55}^{\alpha_{5}}\Sigma_{66}^{\alpha_{5}}}{D_{77}^{\beta_{2}}D_{88}^{\beta_{2}}D_{88}^{\beta_{3}}\Sigma_{66}^{\beta_{4}}D_{55}^{\beta_{5}}\Sigma_{66}^{\beta_{5}}}\prod_{j=1}^{8}\left(\Sigma_{jj|\prec
j\succ}\right)^{-\frac{pa_{j}+2}{2}}\det\left(\Sigma_{\prec j \succ}\right)^{-\frac{1}{2}}\\
&=D_{11}^{\alpha_{5}}D_{22}^{\alpha_{4}}D_{33}^{\alpha_{2}}D_{44}^{\alpha_{1}}D_{55}^{\alpha_{4}+\alpha_{5}-\beta_{5}}D_{66}^{\alpha_{3}}D_{77}^{\alpha_{1}+\alpha_{2}-\beta_{2}}D_{88}^{\alpha_{1}+\alpha_{2}+\alpha_{3}-\beta_{2}-\beta_{3}}\Sigma_{66}^{\alpha_{4}+\alpha_{5}-\beta_{4}-\beta_{5}}
\prod_{j=1}^{8}\left(\Sigma_{jj|\prec j\succ}\right)^{-\frac{pa_{j}+2}{2}}\det\left(\Sigma_{\prec j \succ}\right)^{-\frac{1}{2}}.
\end{align}
Let  $\lambda_{j}$  be the exponent of $D_{jj}$  in Equation \eqref{eq:markov_ratio_example_b}  for each  $j=1,\ldots, 8$.  If we set  $\alpha_{4}+\alpha_{5}-\beta_{4}-\beta_{5}=0$,  then we obtain
\begin{equation}\label{eq:LM_andersson}
\int_{\Q_{\G}}\omega_{\Q{\G}}\left(\alpha,\beta, U,d\Sigma^{E}\right)=\int_{\Q_{\G}}\exp\left\{-\tr\left(\Sigma U\right)\right\}\prod_{j=1}^{8}D_{jj}^{\lambda_{j}}
\prod_{j=1}^{8}\left(\Sigma_{jj|\prec j\succ}\right)^{-\frac{pa_{j}+2}{2}}\det\left(\Sigma_{\prec j \succ}\right)^{-\frac{1}{2}}d\Sigma^{E}.
\end{equation}
The integrand in the right-hand-side of Equation \eqref{eq:LM_andersson} corresponds to the non-normalized  density of the generalized  Riesz distribution on  $\Q_{\D}$  and  has a finite integral  if and only if $\lambda_{j}>pa_{j}/2$  for each  $j=1,\ldots, 8$.  Furthermore, under these conditions, we have 
\begin{equation}\label{eq:riesz_normalizing_constant}
\int_{\Q_{\G}}\exp\left\{-\tr\left(\Sigma U\right)\right\}\prod_{j=1}^{8}D_{jj}^{\lambda_{j}}
\prod_{j=1}^{8}\left(\Sigma_{jj|\prec j\succ}\right)^{-\frac{pa_{j}+2}{2}}\det\left(\Sigma_{\prec j \succ}\right)^{-\frac{1}{2}}d\Sigma^{E}\propto\prod_{j=1}^{p}\left(U_{jj|\prec j\succ}\right)^{\lambda_{j}}.
\end{equation}
See \cite[ section 6]{A10} for details.  On the other hand, by very similar calculations as those which lead to Equation \eqref{eq:markov_ratio_example_b} we  can show that
\[
H_{\G}\left(\alpha,\beta, U^{E}\right)=\prod_{j=1}^{p}\left(U_{jj|\prec j\succ}\right)^{\lambda_{j}}U_{66}^{\alpha_{4}+\alpha_{5}-\beta_{4}-\beta_{5}}.
\]
Therefore, from Equation \eqref{eq:riesz_normalizing_constant},  for any  $\left(\alpha,\beta\right)$  satisfying   $\lambda_{j}\left(\alpha, \beta\right)>pa_{j}/2$  and   $\alpha_{4}+\alpha_{5}-\beta_{4}-\beta_{5}=0$  we conclude that 
\[
\int_{\Q_{\G}}\omega_{\Q{\G}}\left(\alpha,\beta, U,d\Sigma^{E}\right)/H_{\G}\left(\alpha, \beta, U^{E}\right)<\infty\:\text{and is functionally independent of }\: U^{E}. 
\]
Consequently,  the dimension of  $\mathcal{A}$  is greater than or equal to  $8$. Thus the LM conjecture (I) fails.
\end{Ex}
\section{Closing remarks}

In this paper we develop appropriate tools in order to carefully compare the Wishart Type II distributions introduced by Letac and Massam in \cite{L07} for decomposable graphs and the DAG Wisharts introduced by the authors in \cite{BR11}. The comparison is made when the DAG Wisharts are restricted to the class of perfect DAGs, that is where DAGs are Markov equivalent to the
class of decomposable graphs. By this comparison, we establish the fact that in general, the family of Type II Wisharts is a subfamily of that of DAG Wisharts when the multi-parameters
are restricted to a well identified set $\mathcal{B}_{\mathcal{P}}$. In case of homogeneous graphs we show that the latter
restriction is not needed. In light of this result we are led to a condition on the structure of the graphs that yield counterexamples to the second part the LM conjecture. By taking a similar approach we are also able to reject the first part of the LM conjecture and therefore completely resolve the LM conjecture. 


\end{document}